\documentclass[11pt,oneside]{amsart}
\usepackage{amsmath,amsthm,amssymb,amscd,fancybox,esint}
\usepackage[a4paper,left=3cm,right=3cm,top=3cm,bottom=3cm]{geometry}
\usepackage{cite}
\renewcommand{\epsilon}{\varepsilon}
\renewcommand{\phi}{\varphi}
 \newcommand{\bZ}{\mathbb{Z}}
\newcommand{\bR}{\mathbb{R}} \newcommand{\bN}{\mathbb{N}}
\newcommand{\cF}{\mathcal{F}} 
\newcommand{\dS}{\mathcal{S}(\bR^n)} 

\newcommand{\qp}{\mathcal{Q}}

\newcommand{\Zn}{\mathbb{Z}^n}
\newcommand{\Rn}{\mathbb{R}^n}

\newcommand{\brac}[1]{\langle #1\rangle}
\newtheorem{theorem}{Theorem}[section]
\newtheorem{lemma}[theorem]{Lemma}
\newtheorem{proposition}[theorem]{Proposition}

\newtheorem{corollary}[theorem]{Corollary}
\theoremstyle{definition}
\newtheorem{definition}[theorem]{Definition}
\newtheorem{example}[theorem]{Example}
\theoremstyle{remark}
\newtheorem{remark}[theorem]{Remark}
\numberwithin{equation}{section}

\def\bZ{{\mathbb Z}}

\def\bN{{\mathbb N}}

\def\bC{{\mathbb C}}
\def\bR{{\mathbb R}}

\def\cC{\mathcal{C}}
\def\cD{\mathbb{R}^n}

\def\cF{\mathcal{F}}

\def\cP{\mathcal{P}}
\def\cQ{\mathcal{Q}}

\def\cS{\mathcal{S}}

\newcommand{\C}{\mathbb{C}}
\newcommand{\R}{\mathbb{R}}
\newcommand{\vf}{\mathbf{f}}
\newcommand{\vg}{\mathbf{g}}

\newcommand{\vx}{\mathbf{x}}
\newcommand{\vc}{\mathbf{c}}

\newcommand{\vs}{\mathbf{s}}
\newcommand{\vy}{\mathbf{y}}
\newcommand{\vz}{\mathbf{z}}
\newcommand{\vt}{\mathbf{t}}
\providecommand{\abs}[2][]{#1\lvert#2#1\rvert}
\providecommand{\norm}[2][]{#1\lVert#2#1\rVert}
\def\cprime{$'$} 
\begin{document}%\nocite{*}
\title[Matrix weighted modulation spaces]{ Matrix weighted modulation spaces}
\author[M.\ Nielsen]{Morten Nielsen}
\address{Department of Mathematical Sciences\\ Aalborg
  University\\ Skjernvej 4A\\ DK-9220 Aalborg East\\ Denmark}
\email{mnielsen@math.aau.dk}
\begin{abstract}
 Given a matrix-weight $W$ in the Muckenhoupt class $\mathbf{A}_p(\bR^n)$, $1\leq p<\infty$, we    introduce corresponding vector-valued continuous and discrete $\alpha$-modulation  spaces $M^{s,\alpha}_{p,q}(W)$ and $m^{s,\alpha}_{p,q}(W)$ and prove their equivalence through the use of adapted tight frames. Compatible notions of molecules and almost diagonal matrices are also introduced, and an application to the study of pseudo-differential operators on vector valued spaces is given.

\end{abstract}
\maketitle

\section{Introduction}
The matrix-weighted $L^p$-space $L^p(W)$, $1\leq p<\infty$, are defined
for $W\colon
\cD\to\C^{N\times N}$ a measurable matrix-valued weight function that is positive definite a.e., as  the family of
measurable functions $\vf\colon \cD\to\C^N$ satisfying 
\begin{equation}\label{eq:lp}
  \norm{\vf}_{L^p(W)}:=\Biggl(\int_{\cD}\abs{W^{1/p}(x)\vf(x)}^p\,dx\Biggr)^{1/p}<\infty.
\end{equation}
Factorizing over 
 $$\{\vf\colon \cD\to\C^N;\norm{\vf}_{L^p(W)}=0\},$$
turns $L^p(W)$ into a Banach space. 
These weighted spaces of vector-valued functions have attracted a great deal of  attention recently (see, e.g., \cite{MR4454483,NieSik21,MR3687948,MR2737763,Rou03a}) partly due to the fact that the setup  generates a number interesting mathematical questions related to vector valued functions that is naturally  connected to various classical results on Muckenhoupt weights in harmonic analysis. A highlight 
in the matrix-weighted case is the formulation of a suitable matrix $A_p$ condition by Nazarov, Treil and Volberg that completely characterizes boundedness of the Riesz-transform(s) on $L^p(W)$ for $1<p<\infty$, see \cite{Vol97a,TreVol97a}.

From a more applied point of view, the matrix weighted setup is also of interest due to the fact that the inherent flexibility obtained by varying  properties of the weight function, including adjustment of the size $N$, allows one to adapt the setup to be useful for applications in a variety of mathematical modeling scenarios. One area  where weighted function spaces can be useful  is in the study of partial differential equations with some sort of degeneracy, e.g., ''perturbed'' elliptic  equations with various types of singularities in the coefficients, where it is natural to look for solutions in weighted smoothness spaces, see \cite{MR1410258,MR1469972} and references therein.

 As is well-known, one can use $L^p$-spaces to build a variety of useful smoothness spaces by imposing restrictions on suitable local components of functions measured by a (weighted) $L^p$-norm. Roudenko was the first to apply such an approach in the matrix weighted setup, based on $L^p(W)$ spaces as defined in Eq.\ \eqref{eq:lp}, see \cite{Rou03a}, where she introduced a very natural notion of matrix weighted Besov spaces $B^s_{p,q}(W)$. This work was later extended by Frazier and Roudenko  \cite{Frazier:2004ub, MR4263690} to matrix-weighted Triebel-Lizorkin spaces. 

 Besov  spaces are created by measuring $L^p$-norms of local components of functions corresponding to a dyadic decomposition of the frequency space. However, it was  observed in the scalar case by Triebel \cite{MR0725159} that the same general decomposition approach, using other partitions of the frequency space, can yield other types of useful smoothness spaces such as modulation spaces that are associated with  a uniform decomposition of the frequency space.

The main contribution of the present paper is to present a construction of matrix weighted $\alpha$-modulation space for weights that satisfy a certain matrix Muckenhoupt condition.  The scalar-value $\alpha$-modulation spaces $M^{s,\alpha}_{p,q}(\bR^n)$ are a family of smoothness spaces based on polynomial type decompositions of the frequency space. The family contains the Besov spaces, and the
modulation spaces introduced by Feichtinger \cite{fei}, as special
``endpoint'' cases. The family offers the flexibility to tune general time-frequency properties measured by the smoothness norm by adjusting the parameter $\alpha$ as it determines the structure  of the polynomial
decomposition of the frequency space $\bR^n$ used to define
the corresponding smoothness space. 

The (scalar) $\alpha$-modulation spaces were introduced by Gr\"obner
\cite{Groebner1992} using a general framework of decomposition type
Banach spaces introduced by Feichtinger and Gr\"obner in
\cite{MR87b:46020,MR89a:46053}.  To the best of the author's knowledge,  $\alpha$-modulation spaces have not yet been considered with weights,
so even in the scalar case (i.e., $N=1$) the results presented in the present paper are new.

Scalar (unweighted) $\alpha$-modulation spaces have proven useful in the study of classes of pseudo-differential equations with symbols in certain adapted H\"ormander classes, see \cite{MR2476899,MR2292720,MR2423282,MR3106727}, so there is ample reason to believe that weighted  $\alpha$-modulation spaces can play a role in the study of, e.g., perturbed elliptic equations with  singularities in the coefficients.

The structure of the paper is as follows. In Section \ref{sec:1}, we first recall the time-frequency structure  of scalar-valued $\alpha$-modulation spaces and proceed in Section \ref{sec:22} to extend the definition to obtain  (quasi-)Banach spaces in a certain matrix weighted vector-valued setting. Section \ref{sec:ms} is devoted to obtaining a full discrete characterisation of matrix weighted  $\alpha$-modulation spaces using a simple adapted band-limited frame for the matrix weighted  $\alpha$-modulation spaces. An algebra of discrete almost diagonal matrices adapted  to the matrix weighted  $\alpha$-modulation spaces is introduced in Section \ref{sec:4}, which allows us to define a natural notion of molecules for matrix weighted  $\alpha$-modulation spaces. The almost diagonal matrices and molecules may be used to simplify the study of various operators  on the matrix weighted smoothness spaces, making it much easier to obtain various boundedness results for, e.g., partial differential operators. As an example of the almost diagonal approach, we conclude the paper in Section \ref{sec:5} with a study of Fourier multipliers on matrix weighted  $\alpha$-modulation spaces.

\section{Vector-valued smoothness spaces}\label{sec:1}
In this section we extend the definition of  $\alpha$-modulation spaces to obtain  (quasi-)Banach spaces in a matrix weighted vector-valued setting in a way such  that 
Roudenko's matrix weighted Besov spaces \cite{Rou03a} becomes a special ''endpoint case'' corresponding to $\alpha=1$.
The scalar
$\alpha$-modulation spaces form a family of smoothness spaces that contain  modulation and Besov spaces as special limit  cases. 

The spaces
are defined by imposing restrictions on local components of functions defined using  from specific decompositions of the frequency space. Specifically, the general structure of the local components is governed  by a parameter $\alpha$, belonging  to the interval
$[0,1]$. This parameter determines a segmentation of the frequency
domain from which the spaces are built. We will use the same type of decompositions in the vector-valued setting.

\subsection{The family of $\alpha$-coverings of the frequency domain}\label{sec:mod_space}

We first recall the notion of an $\alpha$-covering as introduced in \cite{Groebner1992,MR87b:46020}.

\begin{definition}\label{def:cov}
A countable collection $\qp$ of measurable subsets $Q\subset \bR^n$ is called an
admissible covering of $\bR^n$ if 
\begin{itemize}
\item[i.]$\bR^n=\cup_{Q\in \qp} Q$ 
\item [ii.] There
exists $n_0<\infty$ such that 
  $\#\{Q'\in\qp:Q\cap Q'\not=\emptyset\}\leq n_0$ for all $Q\in\qp$.
\end{itemize}  
    An
  admissible covering is called an $\alpha$-covering, $0\leq
  \alpha\leq 1$, of $\bR^n$ if
 \begin{itemize}
\item[iii.] 
  $|Q|\asymp \langle \xi\rangle^{\alpha n}$ (uniformly) for all $\xi\in Q$ and for all $Q\in\qp$,
\item [iv.] There exists a constant $ K <\infty$ such that 
$$\sup_{Q\in\qp} \frac{R_Q}{r_Q}\leq K,$$
where $r_Q :=\sup\{r\in [0,\infty):\exists c_r\in\bR^n: B(c_r,r)\subseteq Q\}$
and $R_Q :=\inf\{r\in (0,\infty):\exists c_r\in\bR^n: B(c_r,r)\supseteq Q\}$, where $B(x,r)$ denotes the Euclidean ball in $\bR^n$ centered at $x$ with radius $r$.
\end{itemize}
\end{definition}
\begin{remark}
For $ a  \in\bR^n$ and $r_0>0$,  we define  the corresponding cube $R[{a},{r_0}]$ as
\begin{equation}\label{eq:cube}
R[{a},{r_0}]:= a+r_0[-1,1]^n.
\end{equation}
We notice that for $Q\in\qp$, condition iv.\ in Definition \ref{def:cov} ensures that we have the following  containment in cubes,
$$R[\xi_1,r_Q/(2\sqrt{n})]\subseteq Q\subseteq R[\xi_2,R_Q],$$
for some $\xi_1,\xi_2\in Q$. 
\end{remark}

The following example, which can be considered a ``canonical'' $\alpha$-covering, was first considered in \cite{Groebner1992}, see also \cite{Borup2006a}.
\begin{example}\label{ex:cov}
For $\alpha\in [0,1)$, there exists $c_0>0$ such that for any $c_1\geq c_0$, the family of sets
$$B_k^\alpha:=B\big(\xi_k, c_1 r_k\big),\qquad k\in\Zn,$$
with $B(c,r)$ denoting the (open) Euclidean ball of radius $r>0$ centered at $c\in\bR^n$, and 
\begin{equation}\label{eq:rk}
r_k:=\langle k\rangle^{\frac{\alpha}{1-\alpha}},\qquad \xi_k:=kr_k,   \qquad k\in\bZ^n,
    \end{equation}
  with $\langle \xi\rangle:=(1+|\xi|^2)^{1/2}$, $\xi\in\bR^n$,  form an  $\alpha$-covering.
\end{example}

\begin{remark}
For the case $\alpha=1$, which  is not part of the covering families considered in Example \ref{ex:cov}, a dyadic Lizorkin-type covering of $\bR^n$  form an example of a $1$-covering leading to the matrix-valued Besov spaces considered in \cite{Rou03a}. We refer to \cite{Borup2007} for addition information on Lizorkin-type coverings.
\end{remark}

It is known that any pair of $\alpha$-coverings $\cQ=\{Q\}$ and $\cP=\{P\}$ satisfy the following finite overlap condition, see \cite[Lemma B.2]{Borup2006a},
\begin{equation}\label{eq:overl}
\sup_{Q\in \cQ} \#A_Q<+\infty,\qquad \text{where }A_Q:=\{P\in \cP: P\cap Q\not=\emptyset\}.
\end{equation}
We will need a so-called bounded admissible partition of unity adapted to $\alpha$-coverings. For  $f\in L_1(\bR^n)$, we let  $$\mathcal{F}(f)(\xi):=(2\pi)^{-n/2}\int_{\bR^n}
f(x)e^{- i x\cdot\xi}\,dx,\qquad \xi\in\bR^n,$$
denote the Fourier transform, and we use the standard notation $\hat{f}(\xi)=\mathcal{F}(f)(\xi)$. With this normalisation, the Fourier transform extends to a unitary transform on $L^2(\bR^n)$ and we denote the inverse Fourier transform by $\mathcal{F}^{-1}$.

\begin{definition}\label{de:bapu}
Let  $\qp$ be an $\alpha$-covering of $\bR^n$. A corresponding
 bounded admissible partition of unity (BAPU) $\{\psi_Q\}_{Q\in\qp}$ is a family
of smooth functions satisfying
\begin{itemize}
\item[i.] $\text{supp}(\psi_Q)\subset Q$
\item[ii.] $\sum_{Q\in\qp} \psi_Q(\xi)=1$
\item[iii.]  There exists a uniform constant $C$ such that for $Q\in \cQ$ and $\xi_Q\in Q$,
$$|\mathcal{F}^{-1}(\psi_Q)(x)|\leq C|\xi_k|^{n\alpha}(1+|\xi_Q|^\alpha |x|)^{-n-1},\qquad x\in\bR^n.$$
\end{itemize}

\end{definition}

\begin{remark}
The assumption (iii) in Definition \ref{de:bapu} is slightly modified compared to the usual definition of a BAPU in the scalar case, cf.\ \cite{MR87b:46020,Borup2006a}. This is done in order to accomodate the requirements of a matrix setup, where we will need the convolution result in Lemma \ref{le:conv} below to be applicable to the functions from any BAPU.
%For an $\alpha$-covering $\{Q\}_{Q\in\cQ}$ of $\bR^n$ with associated BAPU $\{\psi_Q\}$, %one can verify that there exists $c_1>0$ such that for $Q\in\cQ$ and any $\xi_Q\in Q$,
%$$\supp\{\psi(|Q|^{\alpha n}\cdot -\xi_Q)\}\subseteq B(0,c_1).$$
\end{remark}

The results in Section
\ref{sec:ms}, and the construction of a $\varphi$-transform, rely on the known fact that it is possible to construct
a smooth BAPU with additional structure. For a straightforward construction of such  a BAPU adapted to the $\alpha$-covering considered in Example \ref{ex:cov}, we may take
$\varphi\in C^\infty(\bR^n)$ to be non-negative with $\varphi(\xi)=1$ when $|\xi|\leq 1$ and $\varphi(\xi)=0$ for $|\xi|> \frac{3}{2}$, and put
\begin{equation}\label{eq:b}
\phi_k(\xi):=\varphi\left(\frac{\xi-r_k k}{c_0 r_k}\right),\qquad k\in\Zn,
\end{equation}
where $c_0$ is the constant from Example \ref{ex:cov} and
$r_k$  given in Eq.\ \eqref{eq:rk}.
 We notice that $\phi_k(\xi)=1$ on the sets  $B_k^\alpha$ from Example \ref{ex:cov} forming an $\alpha$-cover, and, moreover, it can be verified that 
 \begin{equation}\label{eq:bapu}
      \psi_k(\xi):=\frac{\phi_k(\xi)}{{\sum_{\ell\in\Zn} \phi_\ell(\xi)}}.  
\end{equation}
forms a corresponding BAPU, which will be verified in Lemma \ref{le:bab} below. 
A ''square-root'' of this  BAPU can be obtained by setting
\begin{equation}\label{eq:theta}
\theta_k^\alpha(\xi)=\frac{\phi_k(\xi)}{\sqrt{\sum_{\ell\in\Zn} \phi_\ell^2(\xi)}}    
\end{equation}
We then have
 $$\sum_{\ell\in\Zn} [\theta_k^\alpha(\xi)]^2=1,\qquad \xi\in\Rn.$$

It is perhaps less obvious that the functions $\{\mathcal{F}^{-1}(\psi_k)\}_k$ are all well-localised as required in Definition \ref{de:bapu}.(iii). We have the following result.

\begin{lemma}\label{le:bab}
The family of functions $\{\phi_k\}_k$ defined in Eq.\ \eqref{eq:b}  satisfies Definition \ref{de:bapu}.
\end{lemma}
\begin{proof}
Property (i) and (ii) in Definition \ref{de:bapu} are straightforward to verify. We turn to property (iii). Let 
$B_k^\alpha$ be defined as in Example \ref{ex:cov}. For $\eta_k\in B_k^\alpha$  we define $g_k$ by $\widehat{g_k}:=\psi_k(r_k\cdot-\eta_k)$, where one can verify that there exist $r>0$ and constants $C_\beta$, $\beta\in(\bN\cup\{0\})^n$, independent of $k$, such that
$$|(\partial^\beta \widehat{g_k})(\xi)|\leq C_\beta \mathbf{1}_{B(0,r)}(\xi),$$
see, e.g., \cite[Proposition A.1]{Borup2006a}. It follows easily that $|g_k(x)|\leq C(1+|x|)^{-n-1}$ with $C$ independent of $k$. We also notice that
$$\mathcal{F}^{-1}(\psi_k)(x)=r_k^n e^{i\eta_k\cdot x}g_k(r_kx).$$
Hence, we   obtain the decay estimate
\begin{equation}\label{eq:dec}
|\mathcal{F}^{-1}(\psi_k)(x)|=r_k^n|e^{i\eta_k\cdot x}g_k(r_kx)|\leq Cr_k^n(1+r_k|x|)^{-n-1},
\end{equation}
with $C$ independent of $k$, and $\{\psi_k\}_k$ therefore satisfies Definition \ref{de:bapu}.
\end{proof}

Let us also recall the  well-know application of BAPUs to an easy construction of tight frames adapted to the $\alpha$-decompositions, see, e.g., \cite{Borup2007}. Put $$Q_k:=Q_k^\alpha:=R[r_k k,ar_k],\qquad k\in\Zn,$$
be an $\alpha$-cover of cubes with $r_k$ defined in \eqref{eq:rk} with $a\geq \max\{2c_0,\pi\sqrt{n}/2\}$, and where $c_0$ is the constant from Example \ref{ex:cov}. Consider the localized trigonometric system given by 
$$e_{k,\ell}(\xi):=(2\pi)^{-n/2}r_k^{-n/2}\mathbf{1}_{Q_0}(r_k^{-1}\xi-k)e^{i\frac{\pi}{a}\ell\cdot (r_k^{-1}\xi-k)},\qquad k,\ell\in\Zn,$$
with $r_k$ defined in Eq.\ \eqref{eq:rk}. 
Then we define the system $\Phi:=\{\varphi_{k,\ell}\}_{k,\ell}$ in the Fourier domain by
\begin{equation}\label{eq:TF}
\hat{\varphi}_{k,\ell}(\xi)=\theta_k^\alpha(\xi)e_{k,\ell}(\xi),
\end{equation}
where $\{\theta_k^\alpha\}_k$ is the system given by Eq.\ \eqref{eq:theta}.
One can verify that
\begin{equation}\label{eq:muk}
\varphi_{k,\ell}(x)=(2a)^{-n/2}r_k^{n/2}e^{i r_k k\cdot x}\mu_k\bigg(\frac{\pi}{a}\ell+r_k x\bigg),   
\end{equation}
with the function $\mu_k$ given by $\widehat{\mu_k}:=\psi_k(r_k \cdot+\xi_k)$. Using an argument similar to the proof of Lemma \ref{le:bab}, one may verify that $\{\mu_k\}$ are uniformly well-localised in the sense that for every $N\in\bN$ there exists $C_N<\infty$, independent of $k$, such that
\begin{equation}\label{eq:decay-mu}
|\mu_k(x)|\leq C_N\big(1+|x|\big)^{-N},\qquad x\in\Rn.
\end{equation}
This in turn implies that
\begin{equation}\label{eq:e1}
    |\varphi_{k,\ell}(x)|\leq C_N (2a)^{-n/2}r_k^{n/2}\big(1+r_k\big|x-x_{k,\ell}\big|\big)^{-N},\qquad x\in\Rn,
\end{equation}
with 
\begin{equation}\label{eq:xk}
    x_{k,\ell}:=\frac{\pi}{a}r_k^{-1}\ell,\qquad k,\ell\in\bZ^n.
\end{equation}
In the frequency domain, we have $\text{supp}(\hat{\varphi}_{k,\ell})\subset B(\xi_k,cr_k)$ for some $c>0$ independent of $k$, which implies that there exist constants $K_N$ such that the localisation
\begin{equation}\label{eq:e2}
   |\hat{\varphi}_{k,\ell}(\xi)|\leq K_N 
   r_k^{-\frac{n}{2}}(1+r_k^{-1}|\xi_{k}-\xi|)^{-N}, \qquad \xi\in\bR^n,
\end{equation}
holds true for $N\in\bN$. It can easily be verified, see \cite{Borup2007}, that $\Phi:=\{\varphi_{k,\ell}\}_{k,\ell}$ forms a tight frame for $L^2(\Rn)$, i.e., we have the identity
\begin{equation}\label{eq:tf}
f=\sum_{k,\ell} \langle f,\varphi_{k,\ell}\rangle \varphi_{k,\ell},\qquad f\in L^2(\Rn),
\end{equation}
where the sum converges unconditionally in $L^2(\Rn)$.

\subsection{Matrix-weighted $L^p$-spaces and Muckenhoupt weights}\label{sec:22}
As mentioned in the introduction, we will need weighted vector-valued $L^p$-spaces for the construction of smoothness spaces considered below. For $1\leq p<\infty$ and $W\colon
\cD\to\C^{N\times N}$ a matrix-valued function, which is measurable
and positive definite a.e., let $L^p(W)$ denote the family of
measurable functions $\vf\colon \cD\to\C^N$ satisfying Eq.\ \eqref{eq:lp} factorised 
over $$\{\vf\colon \cD\to\C^N;\norm{\vf}_{L^p(W)}=0\}.$$ It can be verified that, for $1<p<\infty$, the dual space to $L^p(W)$ is $L^{p'}(W^{-p'/p})$, where $p'$ is the dual exponent to $p$, i.e., $1/p+1/p'=1$, see \cite{Vol97a} for further details.

 An $N\times N$ matrix weight is a locally integrable and positive definite a.e.\ matrix function $W\colon \cD\rightarrow \C^{N\times N}$. It turns out that one needs certain additional properties of the matrix weight in order to prove, e.g., completeness of the smoothness spaces introduced in the sequel.  The matrix Muckenhoupt condition will play an important role.  We say that a matrix weight $W$ satisfied the matrix $A_p$ condition, $1<p<\infty$, provided
\begin{equation}\label{eq:Roudenko}
 [W]_{{\mathbf{A}_p(\cD)}}:=\sup_{Q\in \cQ} \int_Q\left( \int_Q \big\|W^{1/p}(x)W^{-1/p}(t)\big\|^{p'} \frac{dt}{|Q|}\right)^{p/p'} \frac{dx}{|Q|}<\infty,
\end{equation}
where $\cQ$ is the collection of all cubes in $\cD$.
The norm $\|\cdot\|$ appearing in the integral is any matrix norm on the $N\times N$ matrices. In case \eqref{eq:Roudenko} is satisfied, we write $W\in \mathbf{A}_p(\cD)$.
\begin{remark}
The matrix $A_p$-conditions were introduced and studied in \cite{NazTre96a,TreVol97a,Vol97a} using the notion of dual norms. The condition given in \eqref{eq:Roudenko} was shown to be equivalent to the original condition by Roudenko in \cite{Rou03a}, and \eqref{eq:Roudenko} has the advantage of often being more ``operational''. 
\end{remark}

Similar to the scalar case, special care has to be taken to define a matrix Muckenhoupt condition in the endpoint case $p=1$.
Following Frazier and Roudenko \cite{Frazier:2004ub}, we define the matrix $\mathbf{A}_1(\cD)$-class as follows.
\begin{definition} Let $W:\Omega\rightarrow \bC^{N\times N}$ be a matrix weight. We say that $W\in \mathbf{A}_1(\cD)$ provided that
\begin{equation}\label{eq:mA1}
	\|W\|_{\mathbf{A}_1(\cD)}:=\sup_{Q\in\cQ} \text{esssup}_{y\in Q} \frac{1}{|Q|}\int_Q \|W(t)W^{-1}(y)\|\,dt<+\infty.
\end{equation}

\end{definition}

\begin{remark}
    It can be verified that in the scalar case $N=1$, the conditions given by \eqref{eq:Roudenko} and \eqref{eq:mA1}, respectively,  reduce to the corresponding well-known scalar $A_p$ conditions, see \cite{Frazier:2004ub,Rou03a} for further details.
\end{remark}

\subsection{Vector valued modulation spaces}

 Let $m:\bR^d\rightarrow\bC$ be a bounded measurable function (a multiplier). We denote by $m(D)f:=\mathcal{F}^{-1}(m\hat{f})$, the corresponding Fourier multiplier operator, i.e., the convolution of $\mathcal{F}^{-1}(m)$ with $f$.

We denote by $\cS=\cS(\bR^n)$ the Schwartz space of rapidly decreasing, infinitely differentiable functions on $\bR^n$. A function $\varphi\in\cC^{\infty}$ belongs to $\cS(\bR^n)$ when, for every $k\in\bN_0$ with $\bN_0:=\bN\cup\{0\}$, the semi-norms
\begin{equation}\label{Snorm1}
p_{k}(\varphi):=\max_{\alpha\in \bN_0^n: |\alpha|\leq k}\sup_{x\in\bR^n}(1+|x|)^{k}|\partial ^\alpha \varphi(x)|
\end{equation}
are all finite, where we put $|\alpha|:=\sum_{j=1}^n \alpha_j$ for $\alpha\in \bN_0^n$. As is well-known, the semi-norms $\{p_k\}$ turn $\cS$ into a Fr\'{e}chet space.
The dual space $\cS'=\cS'(\bR^d)$ of $\cS$ is the space of tempered distributions. It will also be useful to consider the corresponding 
concepts in a vector setup, where we consider the direct sum  Fr\'{e}chet space $\bigoplus_{j=1}^N\dS$, with  
dual space $\bigoplus_{j=1}^N\mathcal{S}'(\bR^n)$ consisting of $N$-tuples of  tempered distributions.

We are now ready to give the definition of the vector-valued weighted $\alpha$-modulation spaces.
\begin{definition}
Let $W:\bR^n\rightarrow \bC^{N\times N}$ be a matrix-weight, and let $\qp=\{Q\}$ be an $\alpha$-covering with associated BAPU 
   $\{\psi_Q\}_{Q\in\cQ}$ of the type given in Definition \ref{de:bapu}. Let $\xi_Q\in Q$, $Q\in\cQ$.
For $\alpha\in [0,1]$, $s\in\bR$, $1\leq p<\infty$, and $0<q\leq \infty$, we let $M^{\alpha,s}_{p,q}(W)$ denote the collection of all vector-valued distributions 
$\vf=(f_1,\ldots,f_N)^T\in\bigoplus_{j=1}^N\mathcal{S}'(\bR^n) $, such that
$$\|\vf\|_{M^{\alpha,s}_{p,q}(W)}:=\bigg\|\bigg\{|Q|^{s/n}\|\psi_Q(D)\vf\|_{L^p(W)}\bigg\}_Q\bigg\|_{\ell_q}<\infty,$$
with $\psi_Q(D)\vf:=(\psi_Q(D)f_1,\ldots,\psi_Q(D)f_N)^T$ acting coordinate-wise. For $q=\infty$, the $\ell^q$-norm is replaced by the supremum over $Q$.
\end{definition}
Based on the corresponding definition of (un-weighted) scalar $\alpha$-modulation, one may hope that the family $M^{\alpha,s}_{p,q}(W)$ is in fact a (quasi-)Banach space, at least for  "nice" matrix-weights $W$. This turns out to hold for matrix weights in $\mathbf{A}_p(\cD)$, where we have the following result, where it is also shown that up to equivalence of norms, $M^{\alpha,s}_{p,q}(W)$ is independent of the choice of BAPU whenever $W\in\mathbf{A}_p(\cD)$.
\begin{proposition}\label{prop:complete}
Let $1\leq p<\infty$ and $W\in \mathbf{A}_p(\cD)$. For $0<q\leq \infty$ and $s\in\bR$, 
\begin{itemize}
\item[(a)] We have continuous embeddings
$$\bigoplus_{j=1}^N\dS\hookrightarrow M^{\alpha,s}_{p,q}(W)\hookrightarrow \bigoplus_{j=1}^N\mathcal{S}'(\bR^n).$$
    \item[(b)] The space $M^{\alpha,s}_{p,q}(W)$ is complete, i.e., $M^{\alpha,s}_{p,q}(W)$ is a (quasi-)Banach space.
    \item[(c)] The space $M^{\alpha,s}_{p,q}(W)$ is independent of the choice of BAPU (up to equivalence of norms).
\end{itemize}

\end{proposition}
We will postpone the proof of (a) and (b) until Appendix \ref{sec:complete} as we first need to develop a number of technical tools providing estimates to handle certain band-limited vector-valued functions.
To prove (c), we will need the following convolution lemma proven by Frazier and Roudenko \cite[Lemma 4.4]{MR4263690}. 
The reader may also consult Goldberg's result on general singular integrals \cite{Gol03a}  for the range $1<p<\infty$.

\begin{lemma}\label{le:conv}
Let $1\leq  p<\infty$ and  $W\in {\mathbf{A}_p}$. Suppose that $g:\bR^n\rightarrow \bC$ with $|g(x)|\leq C'(1+|x|)^{-n-1}$ for some constant $C'$, and let $g_\delta(x)=\delta^ng(\delta x)$ for $\delta>0$. If $\vf\in L^p(W)$ then $g_\delta*\vf\in L^p(W)$ and
$$\|g_\delta*\vf\|_{L^p(W)}\leq C \|\vf\|_{L^p(W)},$$
for some constant $C:=C(W,C',p)$ independent of $\delta>0$.
\end{lemma}

\begin{remark}
The proof in \cite{MR4263690} covers the discrete cases $\delta=2^j$, $j\in\bZ$,  but the reader can easily verify that the same proof  extends to cover any $\delta>0$.
\end{remark}

\begin{remark}\label{re:mu}
Suppose   $W\in {\mathbf{A}_p}$ for some $1\leq  p<\infty$. Consider a BAPU $\{\psi_Q\}$  
satisfying Definition \ref{de:bapu}   associated with an $\alpha$-covering $\cQ$ of $\bR^n$. For $\xi_Q\in Q\in\cQ$, we may define a well-localised function $g$ by $\hat{g}:=\psi_Q(|\xi_Q|^{\alpha }\cdot-\xi_Q)$. 
Similar to the estimate \eqref{eq:dec}, one may  obtain the localisation
$$|\mathcal{F}^{-1}(\psi_Q)(x)|=|\xi_Q|^{n \alpha }|e^{i\xi_Q\cdot x}g(|\xi_Q|^{\alpha }x)|\leq C|\xi_Q|^{n\alpha}(1+|\xi_Q|^{\alpha }|x|)^{-n-1},$$
with $C$ independent of $Q$. Hence, by Lemma \ref{le:conv}, we finally arrive at the uniform bound
$$\|\psi_Q(D)\vf\|_{L^p(W)}=\|\mathcal{F}^{-1}(\psi_Q)*\vf\|_{L^p(W)}\leq 
C \|\vf\|_{L^p(W)},$$
with $C:=C(W,p)$ independent of $Q$.
\end{remark}

We can now prove Proposition \ref{prop:complete}.(c).
\begin{proof}[{Proof of Proposition \ref{prop:complete}.(c)}]
Let $\cQ=\{Q\}$ and $\cP=\{P\}$ be two $\alpha$-coverings with associated BAPUs
$\Psi=\{\psi_Q\}_{Q\in\cQ}$ and $\Gamma=\{\gamma_P\}_{P\in \cP}$, respectively.
We first notice, using uniformly bounded height of any $\alpha$-covering, that for $Q\in \cQ$,
\begin{equation}\label{eq:pdo}
\psi_Q(D)\vf=\psi_Q(D)\sum_{P\in A_Q}\gamma_P(D)\vf,
\end{equation}
with $A_Q=\{P\in \cP:P\cap Q\not=\emptyset\}$, where we recall that $\#A_Q$ is bounded by a constant $n_0$ independent of $Q$, see Eq.\ \eqref{eq:overl}. Hence, by Lemma \ref{le:conv} and Remark \ref{re:mu},
$$\|\psi_Q(D)\vf\|_{L^p(W)}\leq C\sum_{P\in A_Q}\|\gamma_P(D)\vf\|_{L^p(W)}.$$
By Definition \ref{def:mad}.(iii), for $Q\in\cQ$ and $P\in\cP$ with  $P\cap Q\not=\emptyset$, we have $|Q|\asymp \langle \xi_0 \rangle^{\alpha n}\asymp |P|$ uniformly for any $\xi_0\in Q\cap P$. It follows from this observation that
$$|Q|^{s/n}\|\psi_Q(D)\vf\|_{L^p(W)}\leq C\sum_{P\in A_Q}|P|^{s/n}\|\gamma_P(D)\vf\|_{L^p(W)}.$$
Similarly, we obtain, for $P\in\cP$,
$$|P|^{s/n}\|\gamma_P(D)\vf\|_{L^p(W)}\leq C\sum_{Q\in B_P}|Q|^{s/n}\|\psi_Q(D)\vf\|_{L^p(W)},$$
with $B_P=\{Q\in \cQ:Q\cap P\not=\emptyset\}$.
Using the uniform bounds on the cardinality of the sets $A_Q$ and $B_P$, it is  then straightforward to verify that
\begin{align*}
\|\vf\|_{M^{\alpha,s}_{p,q}(W)}&=\bigg\|\bigg\{|Q|^{s/n}\|\psi_Q(D)\vf\|_{L^p(W)}\bigg\}_Q\bigg\|_{\ell_q}\\&\asymp 
\bigg\|\bigg\{|P|^{s/n}\|\gamma_P(D)\vf\|_{L^p(W)}\bigg\}_P\bigg\|_{\ell_q},\qquad \vf\in 
M^{\alpha,s}_{p,q}(W).
\end{align*}

\end{proof}
The completion of the proof of Proposition \ref{prop:complete} can be found in Appendix \ref{sec:complete}.
To simplify the notation below, we will call on the equivalence provided by Proposition \ref{prop:complete}.(iii) and for $\alpha\in [0,1)$ always use the ``canonical'' BAPU $\{\psi_k\}_{k\in \Zn}$ given in \eqref{eq:bapu} associated with the $\alpha$-covering of Example \ref{ex:cov}. With this choice, for  
 $s\in\bR$, $1\leq p<\infty$, and $0<q\leq \infty$, we have 
\begin{equation}\label{eq:canonical}
\|\vf\|_{M^{\alpha,s}_{p,q}(W)}=\bigg\|\bigg\{r_k^s\|\psi_k(D)\vf\|_{L^p(W)}\bigg\}_k\bigg\|_{\ell_q},\qquad \vf\in 
M^{\alpha,s}_{p,q}(W).
\end{equation}

\begin{remark}
One may use the same reasoning as used for the proof of   Proposition \ref{prop:complete}.(c) to verify that the ``square root'' system $\{\theta_k^\alpha\}$ defined in \eqref{eq:theta}  also satisfies, for  
 $s\in\bR$, $1\leq p<\infty$, and $0<q\leq \infty$,
\begin{equation}\label{eq:canonicaltheta}
\|\vf\|_{M^{\alpha,s}_{p,q}(W)}\asymp\bigg\|\bigg\{r_k^s\|\theta_k^\alpha(D)\vf\|_{L^p(W)}\bigg\}_k\bigg\|_{\ell_q},\qquad \vf\in 
M^{\alpha,s}_{p,q}(W).
\end{equation}
The details are left for the reader.
\end{remark}
\section{Discrete vector valued modulation spaces and norm characterzations}\label{sec:ms}
Often the notion of smoothness can be linked to sparseness for suitable functions expansions.   In this section we define discrete vector-valued weighted $\alpha$-modulation space together with a simple construction of adapted tight frames that will support a $\phi$-transform in the spirit of the classical construction by Frazier and Jawerth \cite{Frazier1985,Frazier1990}.

Let $k,\ell\in \Zn$, and let   $r_k$ be as in Eq.\ \eqref{eq:rk}. Using the notation introduced in \eqref{eq:cube}, we define for a constant $a>\max\{2c_1,\pi\sqrt{n}/2\}$, with $c_1$ the constant from Example \ref{ex:cov}, the  sets
\begin{equation}\label{eq:qk}
Q(k,\ell):=R\bigg[\frac{\pi}{a}r_k^{-1}\ell,\frac{\pi}{a}r_k^{-1}\bigg]=
\frac{\pi}{a}r_k^{-1}\ell+\bigg[0,\frac{\pi}{a}r_k^{-1}\bigg)^n.
\end{equation}
Clearly, for fixed $k$, $\cQ_k:=\cup_\ell Q(k,\ell)$ forms a partition of $\bR^n$ with $|Q(k,\ell)|=\left(\frac{\pi}{a}\right)^nr_k^{-n}$.
The sets will play the role of a suitable substitute for the dyadic cubes, so we denote 
$\cQ=\cup_k \cQ_k$.

%It is easy to verify that there exists $0<L<\infty$ such that
%$$\sum_{\ell\in\Zn} \mathbf{1}_{Q(k,\ell)}(x)\leq L,\qquad x\in\bR^n,k\in\Zn,$$
Let $\mathbf{1}_A$ denote the characteristic function of a measurable set $A$. We have the following definition.

\begin{definition}
Let $W:\bR^n\rightarrow \bC^{N\times N}$ be a matrix-weight, and suppose
$\alpha\in [0,1]$, $s\in\bR$, $1\leq p<\infty$, and $0<q\leq \infty$. We let 
let $\cQ=\{Q(k,\ell)\}_{k,\ell}$ be the collection of sets defined in \eqref{eq:qk}.
We let $m^{\alpha,s}_{p,q}(W)$ denote the collection of all vector-valued sequences 
$\vs=\{\vs_Q\}_{Q\in\cQ}$, where $\vs_Q=\big(s_Q^{(1)}, \ldots,s_Q^{(N)}\big)^T$, enumerated by the sets in $\cQ$, such that
\begin{align*}
\|\{\vs_Q\}_{Q}\|_{m^{\alpha,s}_{p,q}(W)}&:=\bigg\|\bigg\{r_k^s\bigg\|\sum_{\ell\in\bZ^n} |Q(k,\ell)|^{-\frac{1}{2}}\vs_{Q(k,\ell)}\mathbf{1}_{Q(k,\ell)}\bigg\|_{L^p(W)}\bigg\}_k\bigg\|_{\ell_q}\\
&=\bigg(\sum_{k\in\bZ^n}\bigg\| r_k^s\sum_{\ell\in\bZ^n} |Q(k,\ell)|^{-\frac{1}{2}}\big\|W^{1/p}(t)\vs_{Q(k,\ell)}\big\|\mathbf{1}_{Q(k,\ell)}(t)\bigg\|_{L^p(dt)}^q\bigg)^{1/q}.
\end{align*}
 For $q=\infty$, the $\ell^q$-norm is replaced by the supremum over $k$.
\end{definition}

In order to make a connection between the discrete spaces $m^{\alpha,s}_{p,q}(W)$ and the continuous setting, we will rely on a number of weighted sampling results for band-limited vector-valued functions. The following Lemma provides a weighted $L^p$ sampling result for band-limited vector-valued functions that have their frequency support contained in sets that are not ''centered''. Similar to the scalar case, it is possible to center the spectrum by translation, which corresponds to applying a modulation to  the function, which  does not affect the $L^p$-properties considered in the Lemma.   
We let $$\Omega_k:=\{\vf:\bR^n\rightarrow\bC^N|\text{supp}(\hat{f_i})\subseteq B_k^\alpha, i=1,\ldots,N\},$$
where $B_k^\alpha$ are the sets considered in Example \ref{ex:cov}. We have the following lemma, which provides an adapted version of the sampling result \cite[Lemma 6.3]{Rou03a} by Roudenko.
\begin{lemma}\label{le:samp}
Let $1\leq p<\infty$,  $W\in \mathbf{A}_p$, and suppose $\vg\in \Omega_k$. Then there exists a constant $c_{p,n}$, independent of $\vg$, such that
$$\sum_{\ell\in \bZ^n}\int_{Q(k,\ell)}\bigg|W^{1/p}(x)\vg\bigg(\frac{\pi}{a}r_k^{-1}\ell\bigg)\bigg|^p\,dx\leq c_{p,n}\|\vg\|_{L^p(W)}^p.$$
\end{lemma}
\begin{proof}
Recall that $B_k^\alpha:=B(kr_k,c_1r_k)$, so for $\vg\in \Omega_k$, the modified function $$\tilde{\vg}:=e^{i \frac{\pi}{a} k\cdot}\vg\bigg(\frac{\pi}{a}r_k^{-1}\cdot\bigg)$$ satisfies $\text{supp}(\hat{{\tilde{\vg}}})\subseteq B(0,c_1\frac{\pi}{a})\subseteq B(0,2)$.
Now, by a change of variable,

\begin{align*}\int_{Q(k,\ell)}\big|W^{1/p}(x){\tilde{\vg}}\big(\ell\big)\big|^p\,dx&\asymp r_k^{-n}\int_{\ell+[0,1)^n}\bigg|W^{1/p}\bigg(\frac{\pi}{a}r_k^{-1}x\bigg){\tilde{\vg}}\big(\ell\big)\bigg|^p\,du
\end{align*}
We now use the general sampling result \cite[Lemma 6.3]{Rou03a} to deduce that there exists a constant $c_{p,n}$, independent of $g$, such that
$$\sum_{\ell\in \bZ^n}\int_{\ell+[0,1)^n}\bigg|W^{1/p}\bigg(\frac{\pi}{a}r_k^{-1}x\bigg){\tilde{\vg}}\big(\ell\big)\bigg|^p\,dx\leq  c_{p,n}\|{\tilde{\vg}}\|_{L^p(W(\frac{\pi}{a}r_k^{-1}\cdot))}^p.$$
We apply another change of variable to obtain,
\begin{align*}
\sum_{\ell\in \bZ^n}\int_{Q(k,\ell)}\bigg|W^{1/p}(x)\vg\bigg(\frac{\pi}{a}r_k^{-1}\ell\bigg)\bigg|^p\,dx&\asymp
r_k^{-n}\sum_{\ell\in \bZ^n}\int_{\ell+[0,1)^n}\bigg|W^{1/p}\bigg(\frac{\pi}{a}r_k^{-1}x\bigg){\tilde{\vg}}\big(\ell\big)\bigg|^p\,dx\\&\leq  c_{p,n}r_k^{-n}\|{\tilde{\vg}}\|_{L^p(W(\frac{\pi}{a}r_k^{-1}\cdot))}^p\\
&=c_{p,n}\|\vg\|_{L^p(W)}^p.
\end{align*}
\end{proof}
We can use Lemma \ref{le:samp} to obtain the following norm estimate for the canonical expansion coefficients $\vc_{k,\ell}:=\langle \vf,\phi_{k,\ell}\rangle$ associated with the tight frame defined in Eq.\ \eqref{eq:TF}. The result show that the natural analysis operator for this system is bounded from $M_{p,q}^{\alpha,s}(W)$ to $m_{p,q}^{\alpha,s}(W)$ -- at least for "nice" matrix weights $W$.
%\marginpar{{\bf !!!!}}
\begin{proposition}\label{prop:coef}
  Let $\alpha\in [0,1]$, $1\leq p<\infty$, and $0<q<\infty$. Suppose $W\in \mathbf{A}_p$. Then there exists a constant $C:=C(\alpha,q,W)$ such that for $\vf\in M_{p,q}^{\alpha,s}(W)$,
  \begin{equation}
      \|\{\vc_{k,\ell}\}_{k,\ell}\|_{m_{p,q}^{\alpha,s}(W)}\leq C \|\vf\|_{M_{p,q}^{\alpha,s}(W)},
  \end{equation}
  with $\vc_{k,\ell}:=\langle \vf,\phi_{k,\ell}\rangle$.
\end{proposition}
\begin{proof}
Recall that $\phi_{k,\ell}\in \Omega_k$ with
$$\hat{\varphi}_{k,\ell}(\xi)=(2a)^{-n/2}e^{-i\frac{\pi}{a}\ell\cdot k}r_k^{-n/2}\theta_k^\alpha(\xi)e^{i\frac{\pi}{a}\ell\cdot r_k^{-1}\xi},
$$
so, we have
$$\langle \vf,\phi_{k,\ell}\rangle=(2\pi)^{-n/2}e^{-i\frac{\pi}{a}\ell\cdot k}|Q(k,\ell)|^{1/2}\theta_k^\alpha(D)\vf\bigg(\frac{\pi}{a}r_k^{-1}\ell\bigg).$$
Hence, using the observation in \eqref{eq:canonicaltheta},
\begin{align*}
  \|  \{\langle \vf,\phi_{k,\ell}\rangle\}_{k,\ell} \|_{m_{p,q}^{\alpha,s}(W)}&=
  \bigg\|\bigg\{r_k^s\bigg\|\sum_{\ell\in\bZ^n} |Q(k,\ell)|^{-\frac{1}{2}}\langle \vf,\phi_{k,\ell}\rangle\mathbf{1}_{Q(k,\ell)}\bigg\|_{L^p(W)}\bigg\}_k\bigg\|_{\ell_q}\\
  &\asymp\bigg\|\bigg\{r_k^s\bigg\|\bigg[\sum_{\ell\in \bZ^n}\int_{Q(k,\ell)}\bigg\|W^{1/p}(x)\theta_k^\alpha(D)\vf \bigg(\frac{\pi}{a}r_k^{-1}\ell\bigg)\bigg\|^p dx\bigg]^{1/p}\bigg\}_k\bigg\|_{\ell_q}\\
  &\leq c_{p,q}\big\|\big\{r_k^s \|\theta_k^\alpha(D)\vf\|_{L^p(W)}\big\}_k\big\|_{\ell_q}\\
  &\asymp \|\vf\|_{M_{p,q}^{\alpha,s}(W)},
\end{align*}
where we used the observation in Eq.\ \eqref{eq:canonicaltheta} for the final estimate.
\end{proof}

  We will now prove that the corresponding reconstruction operator for the tight frame defined in Eq.\ \eqref{eq:TF} is bounded from $m_{p,q}^{\alpha,s}(W)$ to $M_{p,q}^{\alpha,s}(W)$ for suitable weights $W$. 

We will  need the notion of a doubling matrix weight for the following result. Using the notation introduced in \eqref{eq:cube}, we have the following definition.

\begin{definition}\label{def:doub}
    We say that the matrix weight $W:\cD\rightarrow \bC^{N\times N}$ satisfies the doubling condition of order $0<p<\infty$ if there is a constant $c$ such that
    for all $\vx,\vy\in\cD$ and $r>0$,
    \begin{equation}\label{eq:doub}
        \int_{R[\vx,2r]} \|W^{1/p}(t)\vy\|^p\,dt\leq c\int_{R[\vx,r]} \|W^{1/p}(t)\vy\|^p\,dt.
    \end{equation}
    Suppose $c =2^\beta$ is the smallest constant for which \eqref{eq:doub} holds, then $\beta$ is called the doubling exponent of $W$.
\end{definition}

\begin{remark}\label{rem:doub}
Notice that \eqref{eq:doub} is stating the condition that the scalar measure $w_{\vy}(t):=\|W^{1/p}(t) \vy\|^p$ is uniformly doubling and not identically zero (a.e.). It is known that whenever $W \in \mathbf{A}_p$, then $w_{\vy}$ is a scalar $A_p$ weight for any $\vy\in\cD$. Morevover,  the $A_p$ constant is bounded by the $\mathbf{A}_p$ constant of $W$ and thus independent of $\vy$, see, e.g., \cite[Corollary 2.3]{Gol03a}. This implies that  $w_{\vy}$ is a scalar doubling measure, see \cite{MR1232192}, and the corresponding doubling exponent $\beta$ is also independent of $\vy$.
 \end{remark} 
 
 We have the following result that in particular applies to matrix weights in $\mathbf{A}_p$, c.f.\ Remark \ref{rem:doub}. 
\begin{proposition}\label{prop:recon}
  Let $\alpha\in [0,1]$, $1\leq p<\infty$, $0<q<\infty$, and suppose $W$ satisfies \eqref{eq:doub}. Then there exists a constant $C$ such that for any finite vector-valued coefficient sequence  $\vs:=\{\vc_{j,\ell}\}_{(j,\ell)\in F}$, $F\subset \bZ^n\times\bZ^n$,
  \begin{equation}
      \bigg\|\sum_{(j,\ell)\in F}\vc_{j,\ell}\phi_{j,\ell}\bigg\|_{M_{p,q}^{\alpha,s}(W)}\leq C \|\{\vc_{j,\ell}\}\|_{m_{p,q}^{\alpha,s}(W)}.
  \end{equation}
  
\end{proposition}
\begin{proof}
We have, using the fact that $\text{supp}(\psi_k)\subseteq B^\alpha_k$,

\begin{align*}
      \bigg\|\sum_{(j,\ell)\in F}\vc_{j,\ell}\phi_{j,\ell}\bigg\|_{M_{p,q}^{\alpha,s}(W)}&=
      \bigg\|\bigg\{r_k^s\bigg\|\psi_k(D)\sum_{(j,\ell)\in F}\vc_{j,\ell}\phi_{j,\ell}\bigg\|_{L^p(W)}\bigg\}_k\bigg\|_{\ell_q}\\
      &= \bigg\|\bigg\{r_k^s\bigg\|\psi_k(D)\sum_{j\in N(k)}\sum_\ell\vc_{j,\ell}\phi_{j,\ell}\bigg\|_{L^p(W)}\bigg\}_k\bigg\|_{\ell_q},
\end{align*}
where  $N(k)=\{m\in\bZ^n:B^\alpha_m\cap B^\alpha_k\not=\emptyset\}$. Now, $r_j\asymp r_k$ (uniformly) for $j\in N(k)$, so by Remark \ref{re:mu},

\begin{align*}
r_k^s\bigg\|\psi_k(D)\sum_{j\in N(k)}\sum_\ell\vc_{j,\ell}\phi_{j,\ell}\bigg\|_{L^p(W)}
     & \leq C r_k^s\bigg\|\sum_{j\in N(k)}\sum_\ell\vc_{j,\ell}\phi_{j,\ell}\bigg\|_{L^p(W)}\\
     &\leq C' \sum_{j\in N(k)}\bigg\|r_j^s \sum_\ell\vc_{j,\ell}\phi_{j,\ell}\bigg\|_{L^p(W)}.
\end{align*}
Recall that $\phi_{j,\ell}$ satisfies the decay property \eqref{eq:e1} for any $N>0$. We now use \eqref{eq:e1} to obtain the estimate

\begin{align*}
\bigg\| \sum_\ell\vc_{j,\ell}\phi_{j,\ell}\bigg\|_{L^p(W)}^p&\leq 
\int_{\bR^n} \bigg(\sum_\ell \|W^{1/p}(x)\vc_{j,\ell}\|\, |\phi_{j,\ell}(x)|\bigg)^p\,dx\\
&\leq C_N\int_{\bR^n} \bigg(r_j^{n/2}\sum_\ell \|W^{1/p}(x)\vc_{j,\ell}\|\, \big(1+r_j|x-x_{j,\ell}|\big)^{-N}\bigg)^p\,dx\\
%&=C_Nr_j^{n/2}\int_{\bR^n} \bigg(r_j^{-n/2}\sum_\ell \|W^{1/p}(r_j^{-1}u)\vc_{j,\ell}\|\, \bigg(1+\bigg|u-\frac{\pi}{a}\ell\bigg|\bigg)^{-N}\bigg)^p\,du\\
&\leq C_N'r_j^{np/2}\int_{\bR^n} \sum_\ell \|W^{1/p}(x)\vc_{j,\ell}\|^p\, \big(1+r_j\big|x-x_{j,\ell}\big|\big)^{-\frac{Np}{2}}\,dx,
    \end{align*}
where we used the discrete H\"older inequality for the last step in the case $1<p<\infty$ with $N$ chosen large enough such that for the dual H\"older exponent $p'$ to $p$,
$$
  \sup_{u\in\bR^n} \sum_\ell \bigg(1+\bigg|u-\frac{\pi}{a}\ell\bigg|\bigg)^{-\frac{Np'}{2}} <\infty.$$
  For we $p=1$ we obtain the estimate directly without using H\"older's inequality. 
  The function $w_{j,\ell}(x):=\|W^{1/p}(x)\vc_{j,\ell}\|^p$ is doubling with a doubling constant $\beta>0$ independent of $j$ and $\ell$. We can therefore use Lemma \ref{le:sq} below to obtain the following estimate,
  \begin{align*}
\bigg\| \sum_\ell\vc_{j,\ell}\phi_{j,\ell}\bigg\|_{L^p(W)}^p&\leq 
C_N'r_j^{np/2}\sum_\ell\int_{\bR^n}  \|W^{1/p}(x)\vc_{j,\ell}\|^p\, \big(1+r_j\big|x-x_{j,\ell}\big|\big)^{-\frac{Np}{2}}\,dx\\
&\leq 
C_N'r_j^{np/2}\sum_{\ell\in\bZ^n}\int_{Q(j,\ell)}  \|W^{1/p}(x)\vc_{j,\ell}\|^p\, dx\\
&\asymp\bigg\| \sum_\ell |Q(j,\ell)|^{-1/2} \vc_{j,\ell} \mathbf{1}_{Q(j,\ell)}\bigg\|_{L^p(W)}^p.
\end{align*}
We may now conclude that
\begin{align}
      \bigg\|\sum_{(j,\ell)\in F}\vc_{j,\ell}\phi_{j,\ell}\bigg\|_{M_{p,q}^{\alpha,s}(W)}&\leq C
     \bigg\|\bigg\{  \sum_{j\in N(k)}\bigg\|r_j^s \sum_\ell\vc_{j,\ell}\phi_{j,\ell}\bigg\|_{L^p(W)}\bigg\}_k\bigg\|_{\ell_q}\nonumber\\
     &\leq 
 C  \bigg\|\bigg\{  \sum_{j\in N(k)}r_j^s\bigg\| \sum_\ell |Q(j,\ell)|^{-1/2} \vc_{j,\ell} \mathbf{1}_{Q(j,\ell)}\bigg\|_{L^p(W)}\bigg\}_k\bigg\|_{\ell_q}\nonumber\\  
   &\leq 
 C  \bigg\|\bigg\{  \sum_{j}r_j^s\bigg\| \sum_\ell |Q(j,\ell)|^{-1/2} \vc_{j,\ell} \mathbf{1}_{Q(j,\ell)}\bigg\|_{L^p(W)}\bigg\}_k\bigg\|_{\ell_q}\nonumber\\
 &\leq \|\{\vc_{j,\ell}\}\|_{m_{p,q}^{\alpha,s}(W)},\label{eq:LW}
 \end{align}
 where we used the uniform bound on the cardinality of $N(k)$. This concludes the proof.
 
\end{proof}

\begin{remark}\label{rem:dense}
Since $\{\phi_{k,\ell}\}_{k,\ell}\subset \mathcal{S}(\bR^n)$, it follows easily from Proposition  \ref{prop:coef} and Proposition \ref{prop:recon}, by standard arguments, that $\bigoplus_{j=1}^N\dS$ is dense in $M_{p,q}^{\alpha,s}(W)$ whenever $1\leq p<\infty$, $0<q<\infty$, and $W\in\mathbf{A}_p$.
\end{remark}

The following technical lemma was used in the proof of Proposition \ref{prop:recon}.
\begin{lemma}\label{le:sq}
Let $w:\bR^n \rightarrow (0,\infty)$ be a function satisfying the doubling condition
  \begin{equation*}\label{eq:doub2}
        \int_{R[\vx,2r]}w(t)\,dt\leq c\int_{R[\vx,r]} w(t)\,dt,\qquad \vx\in\bR^n,r>0,
    \end{equation*}
 with doubling exponent $\beta>0$ such that $2^\beta=c$. Let $j,\ell\in\bZ^n$ and let the quantities $Q(j,\ell)$, $r_j$, $x_{j,\ell}$ be defined by Eqs.\ \eqref{eq:qk}, \eqref{eq:rk} and \eqref{eq:xk}, respectively. Then for $L>\beta$, we have 
$$\int_{\bR^n}  w(x)\big(1+r_j\big|x-x_{j,\ell}\big|\big)^{-L}\,dx \leq C\int_{Q(j,\ell)}  w(x)\, \,dx,$$
\end{lemma}
\begin{proof}
We make a partition $\bR^n=\cup_{m=0}^\infty R_m$, where $R_0=Q(j,\ell)$ and the rectangular "annuli" $R_m$, $m\geq 1$, is defined by
$$R_m:=\bigg\{y\in\bR^n: \frac{\pi}{a}2^{m-1}r_j^{-1}\leq |y-x_{j,\ell}|_\infty <  \frac{\pi}{a}2^{m}r_j^{-1}\bigg\}.$$
Then 
\begin{align*}
 \int_{\bR^n}  w(x)\big(1+r_j\big|x-x_{j,\ell}\big|\big)^{-L}\,dx&=
\sum_{m=0}^\infty \int_{R_m}  w(x) \big(1+r_j\big|x-x_{j,\ell}\big|\big)^{-L}\,dx\\
&\leq C \sum_{m=0}^\infty 2^{-mL}\int_{R_m}  w(x)\, \,dx\\
\end{align*}
However, by the doubling property of $w(x)$, noting that $R_m\subseteq \{y: |y-x_{j,\ell}|_\infty < 2^m \frac{\pi}{a}r_j^{-1}\}$, we have
$$\int_{R_m}  w(x)\, \,dx \leq c2^{\beta m}\int_{R_0}  w(x)\, \,dx,$$
so
$$\int_{\bR^n}  w(x)\big(1+r_j\big|x-x_{j,\ell}\big|\big)^{-L}\,dx \leq C' \sum_{m=0}^\infty  2^{(\beta -L)m}\int_{R_0}  w(x)\, \,dx\leq C''\int_{R_0}  w(x)\, \,dx,$$
provided that $L>\beta$.
\end{proof}
\section{Stable expansions and almost diagonal matrices}\label{sec:4}
The band-limited tight frame $\{\phi_{j,k}\}$ provides a nice stable decomposition system for  $M^{\alpha,s}_{p,q}(W)$ whenever $W\in \mathbf{A}_p$. It is, however, desirable to extend the stability results to cover more general systems of localised ``molecules" as this will allow us to study, e.g., boundedness of various operators acting on  $M^{\alpha,s}_{p,q}(W)$. In this section, we will study molecules in a discretised setting using an adapted notion of almost diagonal matrices. 
\subsection{Reducing operators and the connection to scalar spaces}
It is known  that for any matrix weight $W:\cD\rightarrow \bC^{N\times N}$, $1\leq p<\infty$,  and $Q=Q(k,\ell)$ denoting the cubes from \eqref{eq:qk}, there exists a nonnegative-definite matrix $A_Q$ such that for $\vx\in\bR^N$,
$$|A_Q\vx|\asymp \rho_{p,Q}(\vx):=\left(\frac{1}{|Q|}\int_Q |W^{1/p}(t)\vx|^p\,dt\right)^{1/p},$$
with equivalence constants independent of $Q$ and $\vx$. $A_Q$ is referred to as a reducing operator for the norm $\rho_{p,Q}$ on $\bR^N$. Frazier and Roudenko, see \cite{Rou03a,Frazier:2004ub,MR4263690}, made the important observation that reducing operators can be used to make certain connections between the matrix-weighted Besov space and the scalar $\phi$-transform studied by Frazier and Jawerth \cite{Frazier1985}. We will now use a similar approach to study the vector-valued $\alpha$-modulation spaces. We will need the following definition.

\begin{definition}
Let $\alpha\in [0,1]$, $s\in \bR$, $1\leq p<\infty$, and $0<q\leq \infty$, and let $\{A_Q\}_{Q\in\cQ}$ be a family of reducing operators associated with $W$. For any vector-valued sequence $\{\vs_{k,\ell}\}_{k,\ell}$, we define 
$$\|\{\vs_{k,\ell}\}_{k,\ell}\|_{m_{p,q}^{\alpha,s}(\{A_{Q(k,\ell)}\})}:=
\bigg\|\bigg\{ r_k^s\bigg\|\sum_{\ell\in\bZ^n} |Q(k,\ell)|^{-\frac{1}{2}}|A_{Q(k,\ell)}\vs_{k,\ell}|\mathbf{1}_{Q(k,\ell)}\bigg\|_{L^p(dt)}\bigg\}_k\bigg\|_{\ell^q}.$$
\end{definition}

We have the following observation.
\begin{lemma}\label{le:eqi}
Let $\alpha\in [0,1]$, $s\in \bR$, $1\leq p<\infty$, and $0<q\leq \infty$, and let $\{A_Q\}_{Q\in\cQ}$ be a family of reducing operators associated with $W$. For any finite vector-valued sequence $\vs=\{\vs_{k,\ell}\}_{k,\ell}$, we have
$$\|\{\vs_{k,\ell}\}_{k,\ell}\|_{m_{p,q}^{\alpha,s}(W)}\asymp
\|\{\vs_{k,\ell}\}_{k,\ell}\|_{m_{p,q}^{\alpha,s}(\{A_{Q(k,\ell)}\})},$$
with equivalence constants independent of $\vs$.
\end{lemma}
\begin{proof}
\begin{align*}
    \|\{\vs_{k,\ell}\}_{k,\ell}\|_{m_{p,q}^{\alpha,s}(W)}
    &=
    \bigg\|\bigg\{r_k^s\bigg\| \sum_{\ell\in\bZ^n} |Q(k,\ell)|^{-\frac{1}{2}}\big|W^{1/p}(t)\vs_{k,\ell}\big|\mathbf{1}_{Q(k,\ell)}(t)\bigg\|_{L^p(dt)}\bigg\}_k\bigg\|_{\ell^q}\\
    &= \bigg\|\bigg\{ r_k^s\bigg(\sum_{\ell\in\bZ^n} |Q(k,\ell)|^{-\frac{p}{2}}[\rho_{p,Q(k,\ell)}(\vs_{k,\ell})]^p|Q(k,\ell)|\bigg)^{1/p}\bigg\}_k\bigg\|_{\ell^q}\\
     &\asymp \bigg\|\bigg\{ r_k^s\bigg(\sum_{\ell\in\bZ^n} |Q(k,\ell)|^{-\frac{p}{2}}|A_{Q(k,\ell)}\vs_{k,\ell}|^p\int_Q \mathbf{1}_{Q(k,\ell)}(t)\,dt\bigg)^{1/p}\bigg\}_k\bigg\|_{\ell^q}\\
     &=\bigg\|\bigg\{ r_k^s\bigg\|\sum_{\ell\in\bZ^n} |Q(k,\ell)|^{-\frac{1}{2}}|A_{Q(k,\ell)}\vs_{k,\ell}|\mathbf{1}_{Q(k,\ell)}\bigg\|_{L^p(dt)}\bigg\}_k\bigg\|_{\ell^q}\\
     &=\|\{\vs_{k,\ell}\}_{k,\ell}\|_{m_{p,q}^{\alpha,s}(\{A_{Q(k,\ell)}\})}.
\end{align*}
\end{proof}

There is a straightforward connection between the (quasi-)norm given by $\|\,\cdot\,\|_{m_{p,q}^{\alpha,s}(\{A_{Q(k,\ell)}\})}$ and the scalar discrete $\alpha$-modulation space norm $\|\cdot\|_{m^{\alpha,s}_{p,q}}$, defined for a scalar sequence $t:=\{t_{k,\ell}\}$ by
\begin{equation}\label{eq:scalar}
    \|t\|_{m^{\alpha,s}_{p,q}}:=\bigg\|\bigg\{ r_k^s\bigg\|\sum_{\ell\in\bZ^n} |Q(k,\ell)|^{-\frac{1}{2}}t_{k,\ell}\mathbf{1}_{Q(k,\ell)}\bigg\|_{L^p(dt)}\bigg\}_k\bigg\|_{\ell^q},
\end{equation}
where we refer to \cite{Borup2007,Nielsen2012} for further details on the scalar discrete $\alpha$-modulation spaces.
Given a vector-valued sequence $\vs:=\{\vs_{k,\ell}\}_{k,\ell}$, we define the scalar sequence $t:=\{t_{Q(k,\ell)}\}$ by putting
$$t_{k,\ell}:=|A_{Q(k,\ell)}\vs_{k,\ell}|.$$
Then we clearly have
\begin{equation}\label{eq:connect}
    \|\vs\|_{m_{p,q}^{\alpha,s}(\{A_{Q(k,\ell)}\})}=\|t\|_{m^{\alpha,s}_{p,q}}.
\end{equation}

\subsection{Almost diagonal matrices} The identity provided by Eq.\ \eqref{eq:connect} combined  with Lemma \ref{le:eqi} allows us to use operators on the scalar space $m^{\alpha,s}_{p,q}$ to study the vector-valued $m^{\alpha,s}_{p,q}(W)$.  For this purpose, it will be useful to recall the following notion of almost diagonal matrices for the scalar spaces $m^{\alpha,s}_{p,q}$  introduced by Rasmussen and the author in \cite{Nielsen2012}.

\begin{definition}\label{doitmad}%
Assume that $\alpha\in[0,1]$, $s\in \bR$, $0<q< \infty$, and ${p}\in [1,\infty)$. A matrix $\mathbf{A}:=\{a_{(j,m)(k,n)}\}_{j,m,k,\ell\in\bZ^d}$
is called almost diagonal on $m^{s,\alpha}_{{p},q}$ if
there exist $J\geq \frac{n}{\min(1,q)}$ and  $C, \delta>0$ such that
\begin{align*}
|a_{(j,\ell)(k,m)}|
&\le C \omega_{(j,\ell)(k,m)}^s(J), \qquad j,m,k,n\in \Zn,
\end{align*}
where
\begin{align*}
\omega_{(j,\ell)(k,m)}^s(J):=\bigg(\frac{r_k}{r_j}\bigg)^{s+\frac{n}{2}}
\min\bigg(\bigg(\frac{r_j}{r_k}\bigg)^{J+\frac{\delta}{2}},&\bigg(\frac{r_k}{r_j}\bigg)^{\frac{\delta}{2}}\bigg) c_{jk}^\delta(J)\\
&\phantom{C}\times
(1+\min(r_k,r_j)|x_{k,m}-x_{j,\ell}|)^{-J-\delta},
\end{align*}
with
\begin{equation*}
c_{jk}^\delta(J):=\min\bigg(\bigg(\frac{r_j}{r_k}\bigg)^{J+\delta},\bigg(\frac{r_k}{r_j}\bigg)^{\delta}\bigg)
(1+\max(r_k,r_j)^{-1}|\xi_k-\xi_j|)^{-J-\delta}
\end{equation*}
with  $r_k$, $\xi_k$, and $x_{k,n}$ defined in Eqs.\ \eqref{eq:rk} and \eqref{eq:xk}.
We denote the set of almost diagonal matrices on $m^{s,\alpha}_{{p},q}$  by $\textrm{ad}_{{p},q}^{s,\alpha}$.
\end{definition}
\begin{remark}\label{rem:ad}
We mention that  a more symmetric sufficient condition  for the matrix $\mathbf{A}:=\{a_{(j,\ell)(k,m)}\}_{j,m,k,\ell\in\bZ^d}$ to be in $\textrm{ad}_{{p},q}^{s,\alpha}$ is obtained by requiring the existence of  $J>\frac{n}{\min(1,q)}$ and $M>\min\{2J,|s|+n/2\}$ such that,
\begin{align}
    |a_{(j,\ell)(k,m)}|&\leq  C\min\bigg\{\bigg(\frac{r_j}{r_k}\bigg)^{M},\bigg(\frac{r_k}{r_j}\bigg)^{M}\bigg\}(1+\min(r_k,r_j)|x_{k,m}-x_{j,\ell}|)^{-J}\nonumber\\
&\hspace{2cm}\times (1+\max(r_k,r_j)^{-1}|\xi_k-\xi_j|)^{-J}.\label{eq:ad_symm}
\end{align}   
\end{remark}
Any matrix $\mathbf{A}=\{a_{(j,\ell)(k,m)}\}_{j,m,k,\ell\in\bZ^d}$ in $\textrm{ad}_{{p},q}^{s,\alpha}$ induces a linear operator on $m^{\alpha,s}_{p,q}$ by calling on the usual matrix vector product. Specifically, let $s\in m^{\alpha,s}_{p,q}$ be a finite sequence, and put $t=\mathbf{A}s$, i.e.,
$$t_{(j,\ell)}:=\sum_{(k,m)\in\bZ^n\times\bZ^n} a_{(j,\ell)(k,m)}s_{(k,m)},\qquad (j,\ell)\in \bZ^n\times\bZ^n.$$
It was proven in \cite{Nielsen2012} that almost diagonal matrices are in fact bounded on the class $m^{\alpha,s}_{p,q}$.

\begin{proposition}\label{johnlennon}%
Suppose that $\mathbf{A} \in \textrm{ad}_{{p},q}^{s,\alpha}$. Then $\mathbf{A}$ is bounded on
$m^{\alpha,s}_{p,q}$.
\end{proposition}

Next, we observe that the following elementary matrix estimate holds 
$$|A_{Q(k,\ell)}\vc|=|A_{Q(k,\ell)}A_{Q(j,m)}^{-1}A_{Q(j,m)}\vc|\leq \|A_{Q(k,\ell)}A_{Q(j,m)}^{-1}\|\cdot |A_{Q(j,m)}\vc|.$$
It is  therefore of interest to study families of reducing operators that are ``compatible'' with the almost diagonal classes introduced. This leads to the following definition.
\begin{definition}
 Let $\{A_Q\}_{Q\in \cQ}$ be a sequence of nonnegative-definite matrices and let $\beta>0$, $1\leq p<\infty$. We say that $\{A_Q\}_{Q\in \cD}$ is strongly doubling of order $(\beta,p)$ if there exists $c>0$ such that for $P=Q(k,m)$ and $Q=Q(j,\ell)$,
\begin{equation}\label{eq:AQ}
    \|A_QA_P^{-1}\|\leq c\max\bigg\{\left(\frac{r_j}{r_k}\right)^{n/p},\left(\frac{r_k}{r_j}\right)^{(\beta-n)/p}\bigg\}
\big(1+\min\{r_j,r_k\}|x_{j,\ell}-x_{k,m}|\big)^{\beta/p}.\end{equation}
\end{definition}

We have the following lemma, where we recall that, as an important special case,  any $W\in \mathbf{A}_p$ is doubling of order $p$.
\begin{lemma}\label{le:dou}
Let     $W$ be a doubling matrix weight of order $p > 0$ with doubling exponent $\beta$ as specified in Definition \ref{def:doub} and suppose $\{ A_Q\}_{Q\in\cQ}$ is a sequence of reducing operators of order $p$ for $W$. Then $\{A_Q\}$ is strongly doubling of order $(\beta,p)$.
\end{lemma}
\begin{proof}
Let $Q=Q(j,\ell), P=Q(k,m)\in\cQ$, and let $\alpha\geq 1$ be a minimal constant for which
$Q\subset \alpha P$. By an elementary geometric estimate, we  have
\begin{equation}\label{eq:alp}
\alpha\leq c \max\bigg\{1,\frac{r_k}{r_j}\bigg\}\big(1+\min\{r_j,r_k\}|x_{j,\ell}-x_{k,m}|\big).
\end{equation}
By the doubling property,
$$w_\vy(Q)\leq w_\vy(\alpha P)\leq c\alpha^\beta w_\vy(P).$$
Hence,
$$|A_Q\vy|^p\leq \frac{1}{|Q|}\int_Q |W^{1/p}(x)\vy|^p\, dx= \frac{1}{|Q|} w_\vy(Q)
\leq c \frac{1}{|Q|}\alpha^\beta w_\vy(P)=
c\frac{|P|}{|Q|}\alpha^\beta |A_P\vy|^p.$$
Now we put $\vy=A_P^{-1}\vz$ for arbitrary $\vz\in\bR^N$, and recall that $|P|\asymp r_k^{-n}$, $|Q|=r_j^{-n}$. Using estimate \eqref{eq:alp} we obtain,
$$|A_QA_P^{-1}\vz|^p\leq c \left(\frac{r_j}{r_k}\right)^n \max\bigg\{1,\frac{r_k}{r_j}\bigg\}^\beta
\big(1+\min\{r_j,r_k\}|x_{j,\ell}-x_{k,m}|\big)^\beta|\vz|^p,$$
and \eqref{eq:AQ} follows.
\end{proof}
We can now define the class of almost diagonal matrices adapted to the vector-valued sequence space $m^{s,\alpha}_{{p},q}(W)$. According to Remark \ref{rem:doub}, the definition in particular applies to the setup where the matrix weight is in $\mathbf{A}_p$. 
\begin{definition}\label{def:mad}
Let     $W$ be a doubling matrix weight of order $p > 0$ with doubling exponent $\beta$ as specified in Definition \ref{def:doub}. Put $K:=\max\big\{\frac{\beta}{p}, \frac{\beta-n}{p}\big\}$. A matrix $\mathbf{A}:=\{a_{(j,\ell)(k,m)}\}_{j,m,k,\ell\in\bZ^d}$
is called almost diagonal on $m^{s,\alpha}_{{p},q}(W)$  if
there exist $J>\frac{n}{\min(1,q)}$, $M>\max\{2J,|s|+n/2\}$,  and  $C>0$  such that
\begin{align}
    |a_{(j,\ell)(k,m)}|&\leq  C\min\bigg\{\bigg(\frac{r_j}{r_k}\bigg)^{M+K},\bigg(\frac{r_k}{r_j}\bigg)^{M+K}\bigg\}(1+\min(r_k,r_j)|x_{k,m}-x_{j,\ell}|)^{-J-\frac{\beta}{p}}\nonumber\\
&\hspace{2cm}\times (1+\max(r_k,r_j)^{-1}|\xi_k-\xi_j|)^{-J},\label{eq:mad_symm}
\end{align} 
with  $r_k$, $\xi_k$, and $x_{k,n}$ defined in Eqs.\ \eqref{eq:rk} and \eqref{eq:xk}.
We denote the set of almost diagonal matrices on $m^{s,\alpha}_{{p},q}$(W)  by $\textbf{ad}_{{p},q}^{\alpha,s}(W)$.

\end{definition}
The following result provides a fundamental boundedness result for matrices in $\textbf{ad}_{{p},q}^{\alpha,s}$ on the vector-valued sequence space $m^{s,\alpha}_{{p},q}(W)$. 
\begin{proposition}\label{prop:di}
Let $\alpha\in [0,1]$, $s\in \bR$, $1\leq p<\infty$, and $0<q\leq \infty$, and let $W\in\mathbf{A}_p$
with order $p$ doubling exponent $\beta$ as specified in Definition \ref{def:doub}. Let $\{A_Q\}_{Q\in\cD}$ be a family of reducing operators associated with $W$. Suppose $\mathbf{B}:=\{b_{(j,m)(k,n)}\}_{j,m,k,\ell\in\bZ^d}$ is almost diagonal with $\mathbf{B}\in \textbf{ad}^{\alpha,s}_{p,q}(J)$. Then $\mathbf{B}$ is bounded on $m^{\alpha,s}_{p,q}(W)$.
\end{proposition}
\begin{proof}
Let $\vs\in m^{\alpha,s}_{p,q}(W)$ be a finite sequence, and put $\vt=\mathbf{B}\vs$, i.e.,
$$\vt_{(j,\ell)}:=\sum_{(k,m)\in\bZ^n\times\bZ^n} b_{(j,\ell)(k,m)}\vs_{(k,m)},\qquad (j,\ell)\in \bZ^n\times\bZ^n.$$
There are no convergence issues due to the fact that $\vs$ is finite.
Now we define an associated scalar sequence $t=(t_Q)_{Q\in\cD}$ by letting 
$t_{(j,\ell)}=|A_{Q(j,\ell)}\vt_{(j,\ell)}|$. Then we notice that
\begin{align*}
    t_{(j,\ell)}&=|A_{Q(j,\ell)}\vt_{(j,\ell)}|\\
    &=\bigg|A_{Q(j,\ell)}\sum_{(k,m)\in\bZ^n\times\bZ^n} b_{(j,\ell)(k,m)}\vs_{(k,m)}\bigg|\\
    &\leq \sum_{(k,m)\in\bZ^n\times\bZ^n} |b_{(j,\ell)(k,m)}|\cdot|A_{Q(j,\ell)}\vs_{(k,m)}|\\
    &\leq  \sum_{(k,m)\in\bZ^n\times\bZ^n} |b_{(j,\ell)(k,m)}|\cdot\|A_{Q(j,\ell)}A_{Q(k,m)}^{-1}\||A_{Q(k,m)}\vs_{(k,m)}|\\
    &=\sum_{(k,m)\in\bZ^n\times\bZ^n} \gamma_{(j,\ell)(k,m)} s_{(k,m)}, 
\end{align*}
with $s_{(k,m)}:=|A_{Q(k,m)}\vs_{(k,m)}|$ and $\gamma_{(j,\ell)(k,m)}:=|b_{(j,\ell)(k,m)}|\|A_{Q(j,\ell)}A_{Q(k,m)}^{-1}\|$. Hence, using the observation in Eq.\ \eqref{eq:connect}, we have
$$\|\{t_{(j,m)}\}_{(j,m)}\|_{m^{\alpha,s}_{p,q}}=\|\{\vt_{(j,m)}\}_{(j,m)}\|_{m^{\alpha,s}_{p,q}(\{A_Q\})}\asymp \|\{\vt_{(j,m)}\}_{(j,m)}\|_{m^{\alpha,s}_{p,q}(W)},$$
and 
$$\|\{s_{(j,m)}\}_{(j,m)}\|_{m^{\alpha,s}_{p,q}}=\|\{\vs_{(j,m)}\}_{(j,m)}\|_{m^{\alpha,s}_{p,q}(\{A_Q\})}\asymp \|\{\vs_{(j,m)}\}_{(j,m)}\|_{m^{\alpha,s}_{p,q}(W)},$$
where we have used Lemma \ref{le:eqi} for the equivalence. Therefore, to prove the wanted boundedness result, it suffice to verify that
$\Gamma:=(\gamma_{(j,m)(k,\ell)})\in\mathrm{ad}_{p,q}^{\alpha,s}$ since, in the scalar setting, an almost diagonal matrix for $m^{\alpha,s}_{p,q}$ will map $m^{\alpha,s}_{p,q}$ boundedly into $m^{\alpha,s}_{p,q}$ according to Proposition \ref{johnlennon}. We notice that the estimate by Lemma \ref{le:dou}, and the almost diagonal assumption on $\mathbf{B}$ given by \eqref{eq:mad_symm}, ensure that
there exists some $J>\frac{n}{\min(1,q)}$ and $M>\max\{2J,|s|+n/2\}$ such that,
\begin{align*}
    |\gamma_{(j,m)(k,n)}|&\leq  C\min\bigg\{\bigg(\frac{r_j}{r_k}\bigg)^{M},\bigg(\frac{r_k}{r_j}\bigg)^{M}\bigg\}(1+\min(r_k,r_j)|x_{k,n}-x_{j,m}|)^{-J}\nonumber\\
&\hspace{2cm}\times (1+\max(r_k,r_j)^{-1}|\xi_k-\xi_j|)^{-J},
\end{align*}   
and as noticed in Remark \ref{rem:ad}, this implies that $\Gamma\in\mathrm{ad}_{p,q}^{\alpha,s}$. 
\end{proof}

Let us now turn to a first useful application of Proposition \ref{prop:di} to study ``change of frame'' operators. As we will see in Corollaries  \ref{cor:recon} and \ref{cor:coeff} below, we can use the ``change of frame'' operators to extend Propositions \ref{prop:coef} and \ref{prop:recon} to cover much more general expansion system.

 We take  $0\leq \alpha<1$ and let
$\{\phi_{k,n}\}_{k,n\in \bZ^n}$ be the tight frame defined in
\eqref{eq:muk} for the chosen $\alpha\in [0,1)$. It can easily be verified that for any fixed $N,P,L>0$, $\phi_{k,n}$
has the following decay in direct and frequency space,
\begin{align}
&|\phi_{k,m}(x)|\le Cr_k^{\frac{n}{2}}(1+r_k|x_{k,m}-x|)^{-2N}\label{one},\\
&|\hat{\phi}_{k,m}(\xi)|\le C
r_k^{-\frac{n}{2}}(1+r_k^{-1}|\xi_{k}-\xi|)^{-2L-2\tfrac{\alpha}{1-\alpha}P},\label{three}
\end{align}
where $C$ is independent of $k$ and $m$, and as before,
\begin{equation}\label{dotodotodoto}
x_{k,m}=\frac{\pi}{a}r_k^{-1}m,\,\, k,m\in\bZ^n,
\end{equation}
with $r_k$  defined in \eqref{eq:rk}. Let
$\{\psi_{k,n}\}_{k,n\in Z^d}\subset L_2(\R^n)$ be another system with similar decay properties,
\begin{align}
&|\psi_{j,m}(x)|\le Cr_j^{\frac{n}{2}}(1+r_j|x_{j,m}-x|)^{-2N}\label{two},\\
&|\hat{\psi}_{j,m}(\xi)|\le C
r_j^{-\frac{n}{2}}(1+r_j^{-1}|\xi_{j}-\xi|)^{-2L-2\tfrac{\alpha}{1-\alpha}{P}}.\label{four}
\end{align}

The following lemma was proved in \cite{Nielsen2012}.
\begin{lemma}\label{bubbleboy}%
Let $0\leq \alpha<1$. Choose $N,P,L>0$ such that $2N>n$ and $2L+2\tfrac{\alpha}{1-\alpha}\frac{P-n}2>n$. If both systems $\{\eta_{k,n}\}_{k,n\in\bZ^n}$  and $\{\psi_{j,m}\}_{j,m\in\bZ^n}$ satisfy \eqref{two} and \eqref{four}, we have
\begin{align*}
|\langle\eta_{k,n},\psi_{j,m}\rangle|
\le& C
\min\bigg(\frac{r_k}{r_j},\frac{r_j}{r_k}\bigg)^{P}(1+\max(r_k,r_j)^{-1}|\xi_k-\xi_j|)^{-L}\\
&\phantom{C} \times(1+\min(r_k,r_j)|x_{k,n}-x_{j,m}|)^{-N}.
\end{align*}
\end{lemma}

The lemma can be applied to obtain the following result reconstruction bound for any system of "molecules" with the same general structure as the frame $\{\phi_{k,n}\}_{k,n\in \bZ^n}$.

\begin{corollary}\label{cor:recon}
Let $\alpha\in [0,1)$, $s\in \bR$, $1\leq p<\infty$, and $0<q\leq \infty$, and let $W\in\mathbf{A}_p$
with order $p$ doubling exponent $\beta$ as specified in Definition \ref{def:doub}.
Put $K:=\max\big\{\frac{\beta}{p}, \frac{\beta-n}{p}\big\}$ and choose $N,P,L>0$ such that 
$2N>n$ and $2L+2\tfrac{\alpha}{1-\alpha}\frac{P-n}2>n$, and, additionally,
\begin{align*}
     \min\{L,N\}>\frac{n}{\min(1,q)}+\frac{\beta}{p},\qquad P>K+\max\bigg\{
     \frac{2n}{\min(1,q)},|s|+\frac{n}{2}\bigg\}.
\end{align*} 
  If the system   $\{\psi_{j,m}\}_{j,m\in\bZ^n}\subset L_2(\R^n)$ satisfy \eqref{two} and \eqref{four} with parameters $N,P,L$ as specified, then there exists a constant $C$ such that for any finite vector-valued coefficient sequence  $\vs:=\{\vc_{j,\ell}\}_{(j,\ell)\in F}$, $F\subset \bZ^n\times\bZ^n$,
  \begin{equation}
      \bigg\|\sum_{(j,\ell)\in F}\vc_{j,\ell}\psi_{j,\ell}\bigg\|_{M_{p,q}^{\alpha,s}(W)}\leq C \|\{\vc_{j,\ell}\}\|_{m_{p,q}^{\alpha,s}(W)}.
  \end{equation}
\end{corollary}
\begin{proof}
We expand $\vf:=\sum_{(j,\ell)\in F}\vc_{j,\ell}\psi_{j,\ell}$ in the canonical system $\{\phi_{j,\ell}\}$. This yields
$$\vf=\sum_{(k,m)\in \bZ^n\times \bZ^n} (\mathbf{B}\vs)_{(j,m)}\phi_{k,m},$$
with $$\mathbf{B}=(\langle \psi_{j,\ell},\phi_{k,n}\rangle)_{(j,\ell)(k,n)}.$$
By Proposition \ref{prop:di}, $\|\mathbf{B}\vs\|_{m^{\alpha,s}_{p,q}(W)}\leq C_1\|\vs\|_{m^{\alpha,s}_{p,q}(W)}$, and it follows by Proposition \ref{prop:recon} that
$$\|\vf\|_{M^{\alpha,s}_{p,q}(W)}\leq C_2\|\mathbf{B}\vs\|_{m^{\alpha,s}_{p,q}(W)}\leq C_1C_2\|\vs\|_{m^{\alpha,s}_{p,q}(W)}.$$
\end{proof}

 Using a similar type of argument, we can also obtain an estimate for the analysis/coefficient operator for any system with the same general structure as the frame $\{\phi_{k,n}\}_{k,n\in \bZ^n}$.
\begin{corollary}\label{cor:coeff}
Let $\alpha\in [0,1)$, $s\in \bR$, $1\leq p<\infty$, and $0<q\leq \infty$, and let $W\in\mathbf{A}_p$
with order $p$ doubling exponent $\beta$ as specified in Definition \ref{def:doub}.
Put $K:=\max\big\{\frac{\beta}{p}, \frac{\beta-n}{p}\big\}$ and choose $N,P,L>0$ such that 
$2N>n$ and $2L+2\tfrac{\alpha}{1-\alpha}\frac{P-n}2>n$, and, additionally,
\begin{align*}
     \min\{L,N\}>\frac{n}{\min(1,q)}+\frac{\beta}{p},\qquad P>K+\max\bigg\{
     \frac{2n}{\min(1,q)},|s|+\frac{n}{2}\bigg\}.
\end{align*} 
  If the system   $\{\psi_{j,m}\}_{j,m\in\bZ^n}\subset L_2(\R^n)$ satisfy \eqref{two} and \eqref{four} with parameters $N,P,L$ as specified, then we have 
 Then for $\vf\in M_{p,q}^{\alpha,s}(W)$,
  \begin{equation}
      \|\{\vc_{k,\ell}\}_{k,\ell}\|_{m_{p,q}^{\alpha,s}(W)}\leq C \|\vf\|_{M_{p,q}^{\alpha,s}(W)},
  \end{equation}
  with $\vc_{k,\ell}:=\langle \vf,\psi_{k,\ell}\rangle$.
\end{corollary}
\begin{proof}
 By Proposition \ref{prop:coef}, there exists $C_1$ such that for  $\vf\in M_{p,q}^{\alpha,s}(W)$,
 $$ \|\{\vs_{k,\ell}\}_{k,\ell}\|_{m_{p,q}^{\alpha,s}(W)}\leq C_1 \|\vf\|_{M_{p,q}^{\alpha,s}(W)},$$
 with $\vs_{k,\ell}:=\langle \vf,\phi_{k,\ell}\rangle$. We also notice that 
 $$\vf=\sum_{(j,m)\in\bZ^n\times\bZ^n}\vs_{j,m}\phi_{j,m}.$$
 We use this representation of $\vf$ to calculate
$$\vc_{k,\ell}:=\langle \vf,\psi_{k,\ell}\rangle,$$
where we obtain $\vc=\mathbf{B}\vs$, with $$\mathbf{B}=(\langle \phi_{j,m},\psi_{k,\ell}\rangle)_{(j,m)(k,\ell)}.$$
 Hence, by Proposition \ref{prop:di},
 $$\|\{\vc_{k,\ell}\}_{k,\ell}\|_{m_{p,q}^{\alpha,s}(W)}\leq C_2 \|\{\vs_{k,\ell}\}_{k,\ell}\|_{m_{p,q}^{\alpha,s}(W)}\leq C_1C_2 \|\vf\|_{M_{p,q}^{\alpha,s}(W)}.$$
\end{proof}

\section{An application to Fourier multipliers}\label{sec:5}
One key selling point of Definition \ref{def:mad} is that new almost-diagonal class $\textbf{ad}_{{p},q}^{\alpha,s}(W)$ is fully compatible with the already known scalar class
${ad}_{{p},q}^{\alpha,s}$ up to a simple $W$-dependent decay modification specified by the doubling exponent $\beta$ --  at least if we rely on the symmetric version of ${ad}_{{p},q}^{\alpha,s}$ discussed in Remark \ref{rem:ad}.

In particular, calling on Proposition \ref{prop:di}, {\em any} scalar boundedness result relying on the symmetric version of the almost diagonal class  ${ad}_{{p},q}^{\alpha,s}$ will have a simple modification to the matrix valued setting for weights in $\mathbf{A}_p$. The specifics of the modification will depend only on the doubling exponent $\beta$ of the matrix-weight. Let us consider an application to Fourier multipliers.
\subsection{Fourier Multipliers}
 Let $1\leq p<\infty$, $0<q<\infty$, and  fix $\alpha\in [0,1]$.
Suppose that $W\in \mathbf{A}_p$. For $m$  a bounded measurable function on $\Rn$, we can define the associated Fourier multiplier as the operator
$$m(D)f = \mathcal{F}^{-1}(m\hat{f}),$$
which is initially defined and bounded on $L^2(\Rn)$. We may extend $m(D)$ to the vector-setup by letting $m(D)$ act coordinate-wise.  We also notice that $m(D)$ is then defined on a dense subset of $M_{p,q}^{\alpha,s}(W)$, c.f.\ Remark \ref{rem:dense}. We have the following result.
\begin{proposition}
Let $1\leq p<\infty$, $0<q<\infty$, $\alpha\in [0,1)$, and suppose that $W\in \mathbf{A}_p$. Fix $b\in\bR$. Assume that the multiplier function $m:\bR^n\rightarrow \bC$ satisfies the smoothness condition
$$\sup_{\xi\in\Rn}\langle \xi\rangle^{\alpha|\eta|-b}|\partial^\eta m(\xi)|<\infty,$$
for every multi-index $\eta\in (\bN\cup\{0\})^n$.  Then
\begin{equation}\label{eq:bdd}
m(D):M_{p,q}^{\alpha,s+b}(W)\rightarrow M_{p,q}^{\alpha,s}(W).    
\end{equation}
\end{proposition}
\begin{proof}
Calling on Proposition \ref{prop:di} and Corollary \ref{cor:recon}, it is straightforward to verify that it suffices to show that the matrix
\begin{equation}\label{eq:bdda}
 \left\{\big\langle \langle \xi_k\rangle^{-b} m(D) \varphi_{k,\ell},\varphi_{j,m}\big\rangle\right\}\in \textbf{ad}_{{p},q}^s(W).
 \end{equation}

Let $Q_k=B_k^\alpha$ be the $\alpha$-covering from Example \ref{ex:cov}. We first observe, using the compact support properties of the system $\{\phi_{k,\ell}\}$, that $\langle m(D)\varphi_{k,\ell},\varphi_{j,m}\rangle=0$ whenever $Q_k\cap Q_j=\emptyset$. Let us therefore focus on the case
$Q_k\cap Q_j\not=\emptyset$, where we have $r_k\asymp r_j$ (with constants independent of $k$ and $j$). The equivalence $r_k\asymp r_j$ in turn implies that we need to verify, 
\begin{equation}\label{eq:bddb}
\big|\big\langle \langle \xi_k\rangle^{-b} m(D) \varphi_{k,\ell},\varphi_{j,m}\big\rangle\big|\asymp (1-|\ell-m|)^{-J-\delta},\qquad Q_k\cap Q_j\not=\emptyset.
\end{equation}
We have, using Eq.\ \eqref{eq:TF},
\begin{align*}
    \langle m(D)\varphi_{k,\ell},\varphi_{j,m}\rangle&=\int_{\Rn} m(\xi)\theta_{k}^\alpha(\xi)\theta_{j}^\alpha(\xi)e_{k,\ell}(\xi)\overline{e_{j,m}}(\xi)\,d\xi,
\end{align*}
and by the affine change of variable $\xi:=T_ky:=r_k y+\xi_k$,
\begin{align*}
    \langle m(D)\varphi_{k,\ell},\varphi_{j,m}\rangle&=r_k^n\int_{\Rn} m(T_k y)\theta_{k}^\alpha(T_k y)\theta_{j}^\alpha(T_k y)e_{k,\ell}(T_k y)\overline{e_{j,m}}(T_k y)\,dy\\
    &=(2\pi)^{-n}\int_{\Rn} m(T_k y)\theta_{k}^\alpha(T_k y)\theta_{j}^\alpha(T_k y)\\&\phantom{(2\pi)^{-n}aaa}\times \exp\left[i\frac{\pi}{a}\bigg(\big(\ell-\frac{r_k}{r_j}m\big)\cdot y -\frac{r_k}{r_j}m\cdot k +m\cdot j\bigg)\right]\,dy.
\end{align*}
Hence, letting $g_{k,j}(y):=m(T_k y)\theta_{k}^\alpha(T_k y)\theta_{j}^\alpha(T_k y)$, we obtain
$$|\langle m(D)\varphi_{k,\ell},\varphi_{j,m}\rangle|\leq C\left|\cF[g_{k,j}]\bigg(\frac{\pi}{a}\bigg[\frac{r_k}{r_j}m-\ell\bigg]\bigg)\right|.$$
Now we proceed to make a standard decay estimate for the Fourier transform of $g_{k,j}$. Notice that
$\theta_{k}^\alpha(T_k \cdot)\theta_{j}^\alpha(T_k \cdot)$ is $C^\infty$ with support contained in a compact set $\Omega$ that can be chosen independent of $k$ and $j$, which can be verified using the fact that  $r_k\asymp r_j$. Hence, by the Leibniz rule, one obtains for $\beta\in (\bN\cup \{0\})^n$,
\begin{align*}|\partial^\beta [g_{k,j}](\xi)|&\leq C_\beta\mathbf{1}_\Omega(\xi)\sum_{\eta\leq \beta} \partial^\eta[m(T_k\cdot)](\xi),\qquad \xi\in \Rn.
\end{align*}
Then by the chain-rule, recalling that $r_k=\langle \xi_k\rangle^\alpha$,
\begin{align*}|\partial^\beta [g_{k,j}](\xi)|&\leq C_\beta \mathbf{1}_\Omega(\xi)\sum_{\eta\leq \beta} r_k^{|\eta|}[(\partial^\eta m)(T_k\cdot)](\xi)\\
&\leq C\sum_{\eta\leq \beta}r_k^{|\eta|} \mathbf{1}_\Omega(\xi) \langle T_k\xi\rangle^{b-\alpha |\eta|}\\
&\leq C\mathbf{1}_\Omega(\xi)\sum_{\eta\leq \beta}\langle \xi_k\rangle^{\alpha|\eta|} \langle \xi_k\rangle^{b-\alpha |\eta|}.
\end{align*}
Standard estimates now show that for any $N>0$, there exists $C_N<\infty$ such that
$$\langle \xi_k\rangle^{-b}|\langle m(D)\varphi_{k,\ell},\varphi_{j,m}\rangle|\leq C_N (1+|m-\ell|)^{-N},$$
whenever $Q_k\cap Q_j\not=\emptyset$. By the observation in Eq.\ \eqref{eq:bddb}, we may therefore conclude that \eqref{eq:bdda} holds, and consequently, that the multiplier result \eqref{eq:bdd} also holds.
\end{proof}
It is well-know that for $b\in\bR$, the bracket-function $\langle\cdot\rangle^b$ satisfies the estimate
$$|[\partial^\beta \langle\cdot\rangle^b](\xi)|\leq C_\beta \brac{\xi}^{b-|\beta|},\qquad \xi\in\bR^n.$$
for any multi-index $\beta \in (\bN\cup\{0\})^n$.
We therefore have the following easy corollary, where we use  $\langle\cdot\rangle^b\langle\cdot\rangle^{-b}\equiv 1$ for the final norm-equivalence.
\begin{corollary}
Let $1\leq p<\infty$, $0<q<\infty$, $\alpha\in [0,1)$, and suppose that $W\in \mathbf{A}_p$. Fix $b\in\bR$. Then
\begin{equation}\label{eq:bessel}
\langle D\rangle^b:M_{p,q}^{\alpha,s+b}(W)\rightarrow M_{p,q}^{\alpha,s}(W).    
\end{equation}
Moreover, we have the norm equivalence
$$ \|g\|_{M_{p,q}^{\alpha,s+b}(W)}\asymp \|\langle D\rangle^b g\|_{M_{p,q}^{\alpha,s}(W)},$$
 for $g\in M_{p,q}^{\alpha,s+b}(W)$.
%\end{proposition}
\end{corollary}

\appendix
\section{Completeness of the $\alpha$-modulation spaces}\label{sec:complete}
Here we complete the proof of  Proposition \ref{prop:complete}.
We recall that $\bigoplus_{j=1}^N \dS$ denotes the family of vector functions $\vf=(f_1,\ldots,f_N)^T$ with $f_i\in\dS$, $i=1,\ldots,N$. We equip the space with the induced semi-norms
$$p(\vf):=\sum_{j=1}^N p_d(f_j),\qquad\text{with } p_d(f_i):=\sup_{\xi\in\bR^n}  \brac{\xi}^d\sum_{|\eta|\leq d} |\partial^\eta\hat{f}_i(\xi)|.$$
As before, we let $\bigotimes_{j=1}^N \mathcal{S}'(\bR^n)$ denote the corresponding family of vector-valued tempered distributions.

%\begin{proposition}
%Let $1\leq p<\infty$, $0<q\leq \infty$ and $s\in\bR$. Then
%\begin{itemize}
%\item[(a)] We have continuous embeddings
%$$\bigotimes_{j=1}^N\dS\hookrightarrow M^{\alpha,s}_{p,q}(W)\hookrightarrow %\bigotimes_{j=1}^N\mathcal{S}'(\bR^n).$$
%    \item[(b)] The space $M^{\alpha,s}_{p,q}(W)$ is complete, i.e., %$M^{\alpha,s}_{p,q}(W)$ is a (quasi-)Banach space.
%\end{itemize}

%\end{proposition}
\begin{proof}[Completion of the proof of Proposition \ref{prop:complete}]
We first show the embedding $\bigoplus_{j=1}^N\dS\hookrightarrow M^{\alpha,s}_{p,q}(W)$. Let $\vf\in \bigoplus_{j=1}^N\dS$, and put  $w(x):=\|W^{1/p}(x)\|^p$. As $W\in \mathbf{A}_p$, it is known that $w$ belongs to scalar $A_p(\bR^n)$, see \cite[Corollary 2.3]{Gol03a}. We will need the fact that scalar $A_p$-wights have moderate average growth in the sense that
\begin{equation}\label{eq:gr}
    \int_{\bR^n} w(x) \brac{x}^{-n(p+\epsilon)}\,dx<+\infty,
    \end{equation}
    for any $\epsilon>0$, 
see, e.g.,  \cite[Chap.\ IX, Proposition 4.5]{torchinsky_real-variable_1986}.

Let $L>0$. We have, for $d> \max\{np,n\alpha+L\}$,

\begin{align}
\int_{\bR^n} |W^{1/p}(x)	\theta_k^\alpha(D)\vf(x)|^p\,dx&\leq 
\int_{\bR^n} \|W^{1/p}(x)\|^p|	\theta_k^\alpha(D)\vf(x)|^p\,dx\nonumber\\
&= \int_{\bR^n} w(x)|	\theta_k^\alpha(D)\vf(x)|^p\,dx\nonumber\\
&\leq  \|\brac{\cdot}^{d}	\theta_k^\alpha(D)\vf\|_\infty^p\int_{\bR^n} w(x)\brac{x}^{-d}\,dx\nonumber\\
&\leq  C\|\brac{\cdot}^{d}	\theta_k^\alpha(D)\vf\|_\infty^p.\label{eq:wp}
\end{align}
 For the scalar function $f_i\in \dS$ it can be shown (see, e.g., \cite[Prop.\ 4.3.(ii)]{MR4082240}) that for $d> \max\{np,n\alpha+L\}$,
\begin{equation}\label{eq:thet}
\|\brac{\cdot}^d	\theta_k^\alpha(D)f_i\|_\infty\leq c\brac{k}^{\frac{-L}{1-\alpha}}p_d(f_i),
\end{equation}
with $c$ independent of $f_i$. 
Hence, using \eqref{eq:gr} and \eqref{eq:wp}, we obtain
$$\|	\theta_k^\alpha(D)\vf\|_{L^p(W)}\leq c_d\brac{k}^{\frac{-L}{1-\alpha}}p_d(\vf)
=c_dr_k^{-L/\alpha}p_d(\vf).$$ Recall that we may choose $L$ arbitrarily large, and for sufficiently large $L$, we  obtain
$$\|\vf\|_{M^{\alpha,s}_{p,q}(W)}=\|\{r_k^s\|	\theta_k^\alpha(D)\vf\|_{L^p(W)}\}_k\|_{\ell^q}\leq c_dp_d(\vf),$$
 for $d$ suitably large, but independent of $\vf$. This provides the wanted embedding.
 
 We now turn to the embedding  $M^{\alpha,s}_{p,q}(W)\hookrightarrow \bigoplus_{j=1}^N\mathcal{S}'(\bR^n)$. Let us first consider the case $1<p<\infty$,
  Take $\vf\in M^{\alpha,s}_{p,q}(W)$, and let $\boldsymbol{\theta}=(\theta_1,\ldots,\theta_N)^T\in  \bigoplus_{j=1}^N\dS$. Then, using the smooth resolution of the identity $\sum_k \theta_k^\alpha(\xi)^2=1$ from Eq.\ \eqref{eq:theta}, we obtain
 \begin{align*}
     \langle \vf,\boldsymbol{\theta}\rangle_{\bC^N}
     &=\sum_{k\in\bZ^n}  \langle \theta_k^\alpha(D)\vf,\theta_k^\alpha(D)\boldsymbol{\theta}\rangle_{\bC^N}\\
     &=\sum_{k\in\bZ^n}  \langle r_k^{s}W^{1/p}\theta_k^\alpha(D)\vf,r_k^{-s}W^{-1/p}\theta_k^\alpha(D)\boldsymbol{\theta}\rangle_{\bC^N}.
 \end{align*}
 In the case $1\leq q<\infty$, we use H\"older's inequality twice to obtain
 $$\int_{\bR^n} |\langle \vf(x),\boldsymbol{\theta}(x)\rangle_{\bC^N}|\,dx
 \leq
 \|\{r_k^s\|\theta_k^\alpha(D)\vf\|_{L^p(W)}\}_k\|_{\ell^q}
 \|\{r_k^{-s}\|\theta_k^\alpha(D)\boldsymbol{\theta}\|_{L^{p'}(W^{-p'/p})}\}_k\|_{\ell^{q'}},
 $$
 with $q'$ the dual H\"older exponent to $q$.
 It is known that $W^{-p'/p}\in \mathbf{A}_{p'}$, see \cite[Corollary 3.3]{Rou03a}. Hence, we may use the embedding already obtained to conclude that, with $d$ suitably large, $$\|\{r_k^{-s}\|\theta_k^\alpha(D)\boldsymbol{\theta}\|_{L^{p'}(W^{-p'/p})}\}_k\|_{\ell^{q'}}\leq c p_d(\boldsymbol{\theta}),$$  so
 $$\int_{\bR^n} |\langle \vf(x),\boldsymbol{\theta}(x)\rangle_{\bC^N}|\,dx
 \leq c\|\vf\|_{M^{\alpha,s}_{p,q}(W)}p_d(\boldsymbol{\theta}),$$
 which proves the  continuous embedding $M^{\alpha,s}_{p,q}(W)\hookrightarrow \bigoplus_{j=1}^N\mathcal{S}'(\bR^n)$. For $0<q<1$ the proof is similar starting from the estimate
 $$|\langle \vf,\boldsymbol{\theta}\rangle_{\bC^N}|^q\leq \sum_{k\in\bZ^n} r_k^{sq}\|\theta_k^\alpha(D)\vf\|_{L^p(W)}^q\big[\sup_\ell r_\ell^{-s}\|\theta_\ell^\alpha(D)\boldsymbol{\theta}\|_{L^{p'}(W^{-p'/p})}\big]^q.$$
 In case $p=1$, we may adapt the same proof relying on the estimates \eqref{eq:thet} and 
 $$\int_{\bR^n} |\langle \vf(x),\boldsymbol{\theta}(x)\rangle_{\bC^N}|\,dx
 \leq
 \|\{r_k^s\|\theta_k^\alpha(D)\vf\|_{L^1(W)}\}_k\|_{\ell^q}
 \|\{r_k^{-s}\|\theta_k^\alpha(D)\boldsymbol{\theta}\|_{\infty}\}_k\|_{\ell^{q'}}.
 $$

Now we turn to the proof of part (b). Let $\{A_Q\}_{Q\in\cD}$ be a sequence of reducing operators associated with $W$, and suppose  $\{\vf_n\}_n$ is a Cauchy sequence in $M^{\alpha,s}_{p,q}(W)$. The sequence is also Cauchy in the complete space $\bigoplus_{j=1}^N\mathcal{S}'(\bR^n)$ by part (a). Hence, the sequence has a well-defined limit $\vf\in \bigoplus_{j=1}^N\mathcal{S}'(\bR^n)$. By Proposition \ref{prop:coef}, $\vc^n=\{\langle \vf_n,\phi_{j,\ell}\rangle_{j,\ell}\}$ is Cauchy in $m^{\alpha,s}_{p,q}(\{A_Q\})$. In particular,
\begin{align}
\sup_k r_k^s\bigg\|\sum_{\ell\in\bZ^n}& |Q(k,\ell)|^{-\frac{1}{2}}\|A_{Q(k,\ell)}(\vc_{k,\ell}^m-\vc_{k,\ell}^n)\|\mathbf{1}_{Q(k,\ell)}\bigg\|_{L^p(dt)}\nonumber\\
&\leq \bigg\|\bigg\{ r_k^s\bigg\|\sum_{\ell\in\bZ^n} |Q(k,\ell)|^{-\frac{1}{2}}\|A_{Q(k,\ell)}(\vc_{k,\ell}^m-\vc_{k,\ell}^n)\|\mathbf{1}_{Q(k,\ell)}\bigg\|_{L^p(dt)}\bigg\}_k\bigg\|_{\ell^q} \longrightarrow 0,\label{eq:et}
\end{align}
as $m,n\rightarrow \infty$, which shows that $\{A_{Q(k,\ell)}\vc_{k,\ell}^m\}_m$ is Cauchy in $\bC^N$, and since  $A_{Q(k,\ell)}$ is invertible, we have that 
$\{\vc_{k,\ell}^m\}_m$  is Cauchy in $\bC^N$. Define $\vc_{k,\ell}:=\lim_{m\rightarrow\infty} \vc_{k,\ell}^m$, where we  have 
$\vc_{k,\ell}=\langle \vf,\phi_{j,\ell}\rangle$ since $\vf_n\rightarrow \vf$ in $\bigoplus_{j=1}^N\mathcal{S}'(\bR^n)$ and $\phi_{j,k}\in\dS$. An application of Fatou's lemma shows that
$$\|\{\vc_{k,\ell}\}\|_{m^{\alpha,s}_{p,q}(W)}\leq \liminf_n\|\{\vc_{k,\ell}^n\}\|_{m^{\alpha,s}_{p,q}(W)}<+\infty.
$$
We may write, $$\vf=\sum_{j\in\bZ^n}\sum_{\ell\in \bZ^n} \vc_{j,\ell}\phi_{j,\ell},$$
where we notice that the norm estimate from Eq.\ \eqref{eq:LW} can be used to verify that $\vf$ is a locally integrable vector function. We apply Fatou's lemma to the right-hand side of the estimate \eqref{eq:et}, and together with Proposition \ref{prop:recon}, to  verify that
$\vf_n\rightarrow \vf$ in $M^{\alpha,s}_{p,q}(W)$, proving completeness of the space.
\end{proof}

%\bibliographystyle{abbrv}
%\bibliography{new_pdo}

\begin{thebibliography}{10}

\bibitem{Borup2006a}
L.~Borup and M.~Nielsen.
\newblock Banach frames for multivariate {$\alpha$}-modulation spaces.
\newblock {\em J. Math. Anal. Appl.}, 321(2):880--895, 2006.

\bibitem{MR2292720}
L.~Borup and M.~Nielsen.
\newblock Boundedness for pseudodifferential operators on multivariate
  {$\alpha$}-modulation spaces.
\newblock {\em Ark. Mat.}, 44(2):241--259, 2006.

\bibitem{Borup2007}
L.~Borup and M.~Nielsen.
\newblock Frame decomposition of decomposition spaces.
\newblock {\em J. Fourier Anal. Appl.}, 13(1):39--70, 2007.

\bibitem{MR2423282}
L.~Borup and M.~Nielsen.
\newblock On anisotropic {T}riebel-{L}izorkin type spaces, with applications to
  the study of pseudo-differential operators.
\newblock {\em J. Funct. Spaces Appl.}, 6(2):107--154, 2008.

\bibitem{MR4082240}
G.~Cleanthous and A.~G. Georgiadis.
\newblock Mixed-norm {$\alpha$}-modulation spaces.
\newblock {\em Trans. Amer. Math. Soc.}, 373(5):3323--3356, 2020.

\bibitem{MR3106727}
E.~Cordero, A.~Tabacco, and P.~Wahlberg.
\newblock Schr\"{o}dinger-type propagators, pseudodifferential operators and
  modulation spaces.
\newblock {\em J. Lond. Math. Soc. (2)}, 88(2):375--395, 2013.

\bibitem{MR3544941}
D.~Cruz-Uribe, K.~Moen, and S.~Rodney.
\newblock Matrix {$ A_p$} weights, degenerate {S}obolev spaces, and mappings of
  finite distortion.
\newblock {\em J. Geom. Anal.}, 26(4):2797--2830, 2016.

\bibitem{MR1410258}
D.~E. Edmunds and H.~Triebel.
\newblock {\em Function spaces, entropy numbers, differential operators},
  volume 120 of {\em Cambridge Tracts in Mathematics}.
\newblock Cambridge University Press, Cambridge, 1996.

\bibitem{MR89a:46053}
H.~G. Feichtinger.
\newblock Banach spaces of distributions defined by decomposition methods.
  {II}.
\newblock {\em Math. Nachr.}, 132:207--237, 1987.

\bibitem{fei}
H.~G. Feichtinger.
\newblock Modulation spaces of locally compact abelian groups.
\newblock In R.~Radha, M.~Krishna, and S.~Thangavelu, editors, {\em Proc.\
  Internat.\ Conf.\ on Wavelets and Applications}, pages 1--56. Allied
  Publishers, New Delhi, 2003.

\bibitem{MR87b:46020}
H.~G. Feichtinger and P.~Gr\"{o}bner.
\newblock Banach spaces of distributions defined by decomposition methods. {I}.
\newblock {\em Math. Nachr.}, 123:97--120, 1985.

\bibitem{Frazier1985}
M.~Frazier and B.~Jawerth.
\newblock Decomposition of {B}esov spaces.
\newblock {\em Indiana Univ. Math. J.}, 34(4):777--799, 1985.

\bibitem{Frazier1990}
M.~Frazier and B.~Jawerth.
\newblock A discrete transform and decompositions of distribution spaces.
\newblock {\em J. Funct. Anal.}, 93(1):34--170, 1990.

\bibitem{Frazier:2004ub}
M.~Frazier and S.~Roudenko.
\newblock Matrix-weighted {Besov} spaces and conditions of ${A}_p$ type for
  $0<p\leq 1$.
\newblock {\em Indiana Univ. Math. J.}, 53(5):1225--1254, 2004.

\bibitem{MR4263690}
M.~Frazier and S.~Roudenko.
\newblock Littlewood-{P}aley theory for matrix-weighted function spaces.
\newblock {\em Math. Ann.}, 380(1-2):487--537, 2021.

\bibitem{Gol03a}
M.~Goldberg.
\newblock Matrix {$A_p$} weights via maximal functions.
\newblock {\em Pacific J. Math.}, 211(2):201--220, 2003.

\bibitem{Groebner1992}
P.~Gr\"obner.
\newblock {\em Banachr\"aume glatter {F}unktionen und {Z}erlegungsmethoden}.
\newblock PhD thesis, University of Vienna, 1992.

\bibitem{MR3687948}
J.~Isralowitz, H.-K. Kwon, and S.~Pott.
\newblock Matrix weighted norm inequalities for commutators and paraproducts
  with matrix symbols.
\newblock {\em J. Lond. Math. Soc. (2)}, 96(1):243--270, 2017.

\bibitem{MR4454483}
J.~Isralowitz, S.~Pott, and S.~Treil.
\newblock Commutators in the two scalar and matrix weighted setting.
\newblock {\em J. Lond. Math. Soc. (2)}, 106(1):1--26, 2022.

\bibitem{MR2476899}
M.~Kobayashi, M.~Sugimoto, and N.~Tomita.
\newblock On the {$L^2$}-boundedness of pseudo-differential operators and their
  commutators with symbols in {$\alpha$}-modulation spaces.
\newblock {\em J. Math. Anal. Appl.}, 350(1):157--169, 2009.

\bibitem{MR1469972}
V.~A. Kozlov, V.~G. Maz\cprime~ya, and J.~Rossmann.
\newblock {\em Elliptic boundary value problems in domains with point
  singularities}, volume~52 of {\em Mathematical Surveys and Monographs}.
\newblock American Mathematical Society, Providence, RI, 1997.

\bibitem{NazTre96a}
F.~L. Nazarov and S.~R. Tre\u{\i}l\cprime.
\newblock The hunt for a {B}ellman function: applications to estimates for
  singular integral operators and to other classical problems of harmonic
  analysis.
\newblock {\em Algebra i Analiz}, 8(5):32--162, 1996.

\bibitem{MR2737763}
M.~Nielsen.
\newblock On stability of finitely generated shift-invariant systems.
\newblock {\em J. Fourier Anal. Appl.}, 16(6):901--920, 2010.

\bibitem{Nielsen2012}
M.~Nielsen and K.~N. Rasmussen.
\newblock Compactly supported frames for decomposition spaces.
\newblock {\em J. Fourier Anal. Appl.}, 18(1):87--117, 2012.

\bibitem{NieSik21}
M.~Nielsen and H.~\v{S}iki\'{c}.
\newblock Muckenhoupt matrix weights.
\newblock {\em J. Geom. Anal.}, 31(9):8850--8865, 2021.

\bibitem{Rou03a}
S.~Roudenko.
\newblock Matrix-weighted {B}esov spaces.
\newblock {\em Trans. Amer. Math. Soc.}, 355(1):273--314, 2003.

\bibitem{MR1232192}
E.~M. Stein.
\newblock {\em Harmonic analysis: real-variable methods, orthogonality, and
  oscillatory integrals}, volume~43 of {\em Princeton Mathematical Series}.
\newblock Princeton University Press, Princeton, NJ, 1993.

\bibitem{torchinsky_real-variable_1986}
A.~Torchinsky.
\newblock {\em Real-variable methods in harmonic analysis}.
\newblock Number v. 123 in Pure and applied mathematics. Academic Press,
  Orlando, 1986.

\bibitem{TreVol97a}
S.~Treil and A.~Volberg.
\newblock Wavelets and the angle between past and future.
\newblock {\em J. Funct. Anal.}, 143(2):269--308, 1997.

\bibitem{MR0725159}
H.~Triebel.
\newblock Modulation spaces on the {E}uclidean {$n$}-space.
\newblock {\em Z. Anal. Anwendungen}, 2(5):443--457, 1983.

\bibitem{Vol97a}
A.~Volberg.
\newblock Matrix {$A_p$} weights via {$S$}-functions.
\newblock {\em J. Amer. Math. Soc.}, 10(2):445--466, 1997.

\end{thebibliography}

\end{document}